\documentclass[10pt]{amsart}
\usepackage[T1]{fontenc}

\usepackage{amsmath, mathrsfs, amsfonts, amssymb, enumerate, caption, subcaption}
\usepackage[numbers]{natbib}

\usepackage{nameref}
\usepackage{zref-xr,zref-user}
\zxrsetup{toltxlabel}

\usepackage{tikz}
\usetikzlibrary{arrows, calc}
\usetikzlibrary{patterns, shapes, decorations.markings}

\usepackage[colorlinks,citecolor=blue,urlcolor=blue]{hyperref}
\usepackage{pgfplots}
\pgfplotsset{/pgf/number format/use comma,compat=newest}
\usetikzlibrary{external}
\numberwithin{equation}{section}

\numberwithin{equation}{section}
\theoremstyle{definition}
\newtheorem{definition}[equation]{Definition}
\theoremstyle{definition}
\newtheorem{remark}[equation]{Remark}
\theoremstyle{definition}

\theoremstyle{definition}

\theoremstyle{lemma}

\theoremstyle{lemma}
\newtheorem{lemma}[equation]{Lemma}
\theoremstyle{theorem}
\newtheorem{theorem}[equation]{Theorem}
\theoremstyle{proposition}
\newtheorem{proposition}[equation]{Proposition}
\theoremstyle{corollary}
\newtheorem{corollary}[equation]{Corollary}

\theoremstyle{corollary}
\theoremstyle{definition}

\theoremstyle{example}

\theoremstyle{proposition}

\theoremstyle{definition}

\DeclareMathOperator*{\supp}{supp}

\newcommand{\R}{\mathbf{R}}

\newcommand{\E}{\mathbf{E}}
\renewcommand{\P}{\mathbf{P}}

\newcommand{\N}{\mathbf{N}}

\title[Nos\'{e}-Hoover under Brownian heating]{Exponential relaxation of the Nos\'{e}-Hoover Thermostat under Brownian heating}
\author{David P. Herzog}
\address{Department of Mathematics, Iowa State University, Ames, IA 50011}

\begin{document}\maketitle
\begin{abstract}
We study a stochastic perturbation of the Nos\'{e}-Hoover equation (called the Nos\'{e}-Hoover equation under Brownian heating) and show that the dynamics converges at a geometric rate to the augmented Gibbs measure in a weighted total variation distance.  The joint marginal distribution of the position and momentum of the particles in turn converges exponentially fast in a similar sense to the canonical Boltzmann-Gibbs distribution.  The result applies to a general number of particles interacting through wide class of potential functions, including the usual polynomial type as well as the singular Lennard-Jones variety.                                       
\end{abstract}

\section{Introduction}

Efficient sampling from the canonical Boltzmann-Gibbs distribution is a fundamental challenge in molecular dynamics simulation.  The most well-known method used to address this challenge is Markov chain Monte Carlo (MCMC) due to Metropolis, Rosenbluth and Rosenbluth~\cite{Met_53}.  The method is based on random walk proposals to generate unbiased samples from a distribution whose density is known up to a normalization constant.  However, it can lead to high correlation.  Moreover, even if one corrects for high correlation via the Metropolis step in Hybrid Monte Carlo (HMC)~\cite{Dua_87, Horo_91}, the correction can be computationally expensive in high dimensions due to gradient evaluations.

The method of stochastic gradients has been successful in addressing the cost of such evaluations~\cite{Ahn_12, Chen_14, Patterson_13, Welling_11}.  The approach is often based on the design of a stochastic differential equation (SDE) whose stationary distribution (or marginal stationary distribution) is the prescribed canonical measure, e.g. Langevin dynamics~\cite{Chen_14, leimkuhler2015computation}.  In this vein, one such SDE that has gained recent interest is the so-called Nos\'{e}-Hoover equation under Brownian heating~\eqref{eqn:main}~\cite{Jones_11, Ding_14} (see also~\cite{Dett_07} for related systems).  In essence,~\eqref{eqn:main} is the usual Langevin system, but it has been augmented by a fictitious control variable.  

Compared with the usual Langevin dynamics, the introduction of this auxiliary variable serves a few important purposes.  First, stochastic gradients can invite noise into the system that is difficult to control~\cite{Chen_14}.  The auxiliary variable is designed to self-correct for this.  Second, from a statistical mechanics perspective, existing methods to sample from the canonical measure often fail to keep the system temperature (defined as mean kinetic energy) near a desired value.  The new control variable does just this by effectively acting as a thermostat.  Lastly, Langevin dynamics has the tendency to spend too much time at small momentum values, thus taking too long to explore phase values.  When the momentum is small, the control variable in~\eqref{eqn:main} pushes it out, allowing for accelerated spatial sampling.

While the~\eqref{eqn:main} system appears to lead to a desirable dynamical method for sampling from the Gibbs measure in high dimensions, many of the statements made above have only been numerically verified.  As an example, it is rigorously known that~\eqref{eqn:main} is uniquely ergodic with an explicit stationary distribution whose marginal coincides with canonical distribution~\cite{Ding_14}.  This result holds true in spite of the degenerate nature of the equations. Nevertheless, how fast the dynamics approaches stationarity is a practical open problem which is made complicated by the introduction of the thermostat, the very object that leads to the convenient properties above.  Intuitively, the auxiliary variable forces the system to large values, and thus how dissipation balances this effect leading to a convergence rate is delicate.  In this paper, we solve this open problem by showing that the~\eqref{eqn:main} system approaches the augmented stationary distribution exponentially fast in an appropriate weighted total variation distance.  By integrating out the control variable, the joint position and momentum distribution in turn converges exponentially fast to the canonical Boltzmann-Gibbs measure in a similar sense.    

It is important to point out that the main convergence result holds for a wide class of potential functions which we call \emph{normal} below (see Definition~\ref{def:normpot}).  This class includes the usual polynomial nonlinearities as well as the singular, Lennard-Jones interactions.  While the methods of hypocoercivity~\cite{Villani_2009, CG_10, Grothaus_Stilgenbauer_2015} have proved useful in extracting convergence rates for such potentials in the Langevin dynamics case, we follow the perturbation methods developed in~\cite{MSH_02, cooke17:_geomet_ergod_two, HerMat17} and construct an explicit Lyapunov function.  In essence, the type of Lyapunov function exhibited here ensures that return times to large compact sets in space have exponential moments.  Thus with the appropriate support and regularity properties of the Markov transitions, geometric convergence to stationarity follows.    

The organization of the paper is as follows.  In Section~\ref{sec:notation}, we introduce the~\eqref{eqn:main} system, fix notation and terminology, and state the main results to be proved in this paper.  Section~\ref{sec:heuristics} provides heuristic ideas behind the construction of the Lyapunov function.  Section~\ref{sec:lyap}, Section~\ref{sec:sup_con} and the Appendix contain the proofs of the main results.

\section{Notation, Terminology, and Main Results}
\label{sec:notation}

Throughout this paper, we study the following system of SDEs
\begin{align}
\label{eqn:main} \tag{NHB}
 dq_i&= \frac{p_i}{m_i}\, dt  \\
\nonumber dp_i &= -\xi p_i  \, dt- \frac{\gamma}{m_i} p_i  \, dt - \nabla_{q_i} U(q) \, dt + \sqrt{2\gamma k_B T}\,  dB_i \\
\nonumber  d\xi&= \sum_{i=1}^N \frac{|p_i|^2}{a m_i} \, dt  - \frac{ k_B T kN}{a}  \,dt . 
\end{align}
The relations above describe the motion of $N\geq 1$ particles in $\R^k$ with position vector $q=(q_i)_{i=1}^N \in (\R^k)^N$, momentum vector $p=(p_i)_{i=1}^N \in (\R^k)^N$ and mass vector $m=(m_i)_{i=1}^N\in (0,\infty)^N$. Each of the particles is subject to friction ($-\tfrac{\gamma}{m_i} p_i \, dt$), thermal fluctuations ($\sqrt{2\gamma k_B T} dB_i$) and a control mechanism $\xi\in \R$, called the \emph{thermostat}, which enacts a friction-like force on the system.  The positive parameters $k_B, \gamma, T, a$ are the Boltzmann, friction, temperature and auxiliary constants, respectively, while the $B_i$, $i=1,2,\ldots, N$, are mutually independent standard $\R^k$-valued Brownian motions defined on a probability space $(\Omega, \mathcal{F}, \P, \E)$.  The function $U:(\R^k)^N\rightarrow [0, +\infty]$ is the potential, and it encapsulates potential forces on the system as well as any potential interactions between the particles.  Throughout, we will assume that $U$ is \emph{normal}, as in the following definition.   
\begin{definition}
\label{def:normpot}
We call a function $U : (\R^k)^N\rightarrow [0, +\infty]$ \emph{normal} if it satisfies the following conditions:
\begin{itemize}
\item[(A1)] The set $\mathcal{O}= \{ q\in (\R^k)^N \, : \, U(q) < \infty\}$ is non-empty, open and connected. 
\item[(A2)]  For every $n\in \N$, the set $\mathcal{O}_n= \{ q\in (\R^k)^N \, : \, U(q) < n\}$ has compact closure in $(\R^k)^N$.        
\item[(A3)]  $U\in C^\infty(\mathcal{O})$ and $\int_\mathcal{O} \exp(-\beta U(q)) \, dq < \infty$ where $\beta:=1/(k_B T)$.
\item[(A4)]  There exists a constant $\zeta\in (1,2)$ such that for any sequence $\{ z_n\}\subseteq (\R^k)^N$ with $U(z_n)\rightarrow \infty$ as $n\rightarrow \infty$ we have 
\begin{align*}
|\nabla U(z_n)| \rightarrow \infty\,\, \text{ and } \,\, \frac{|\nabla^2 U(z_n)|}{| \nabla U(z_n)|^\zeta} \rightarrow 0  
\end{align*}        
as $n\rightarrow \infty$ where $\nabla^2$ denotes the Hessian operator.  
\end{itemize}
\end{definition}  

\begin{remark}
The concept of an \emph{admissible} potential $U$ was introduced in~\cite{HerMat17} to study relaxation properties of Langevin dynamics under a wide class of potentials.  Such a class includes, for example, the Lennard-Jones singular variety.  There is however very little difference between the conditions satisfied by an admissible and normal potential; that is, $U:(\R^k)^N\rightarrow[0, +\infty]$ is \emph{admissible} if it satisfies conditions (A1)-(A3) and condition
\begin{itemize}
\item[(A4')] For any sequence $\{ z_n\}\subseteq (\R^k)^N$ with $U(z_n)\rightarrow \infty$ as $n\rightarrow \infty$ we have 
\begin{align*}
|\nabla U(z_n)| \rightarrow \infty\,\, \text{ and } \,\, \frac{|\nabla^2 U(z_n)|}{| \nabla U(z_n)|^2} \rightarrow 0  
\end{align*}        
as $n\rightarrow \infty$.  
\end{itemize}
Note that condition (A4') is slightly weaker than (A4) in terms of the asymptotic growth of the Hessian relative to the gradient at large potential energy.  Nevertheless, the class of normal potentials is still very wide and includes the usual, polynomial-type potentials as well as the Lennard-Jones singular type.  For further discussion as well as specific examples of normal potentials, we refer the reader to Section~4 of~\cite{HerMat17}.  For other works where similar conditions on $U$ were employed, refer to~\cite{CG_10, Grothaus_Stilgenbauer_2015, Talay_2002}.               
\end{remark}

For notational simplicity, throughout we let $\mathcal{X}=\mathcal{O}\times (\R^k)^N \times \R$ and use $x$ or $(q,p,\xi)$ to denote a generic point in $\mathcal{X}$.  Similarly, the process solving equation~\eqref{eqn:main} will often be denoted more simply by either $(q(t), p(t), \xi(t))$ or $x(t)$.  We let $\mathcal{B}(\mathcal{X})$ denote the Borel $\sigma$-field of subsets of $\mathcal{X}$.  

Equation~\eqref{eqn:main} has Hamiltonian $H:\mathcal{X}\rightarrow [0, \infty)$ defined by
\begin{align}
\label{eqn:Hdef}
H(q,p, \xi) = \frac{1}{2}\|p\|_m^2 + U(q) + \frac{a\xi^2}{2}
\end{align} 
where 
\begin{align}
\frac{1}{2}\|p\|_m^2 := \frac{1}{2}\sum_{i} m_i^{-1}|p_i|^2
\end{align}
is the kinetic energy.  Note that since $U$ satisfies (A3), we may define a probability measure $\mu$ on $\mathcal{B}(\mathcal{X})$, called the \emph{augmented Gibbs measure}, by   
\begin{align}
\label{eqn:Gibbsdef}
\mu( dq  \,d p\,  d\xi) = \frac{1}{\mathcal{Z}} \exp(- \beta H(q,p, \xi)) \, dr\, dp \, d\xi.
\end{align}  
In the above, $\beta=1/(k_B T)$ and $\mathcal{Z}>0$ is the normalization constant making $\mu$ a probability measure.

We next observe the following lemma giving strong existence and uniqueness of solutions of~\eqref{eqn:main}. 
\begin{lemma}
\label{lem:exist}
For every initial condition $x \in \mathcal{X}$, equation~\eqref{eqn:main} has unique strong solution $x(t)$ which is defined for all finite times $t\geq 0$ almost surely on the state space $\mathcal{X}$.       
\end{lemma}

Lemma~\ref{lem:exist} is an immediate consequence of the main result of this paper (see Theorem~\ref{thm:lyapexist} below), but it can be established more easily by using the Hamiltonian $H:\mathcal{X}\rightarrow [0, \infty)$ defined in~\eqref{eqn:Hdef} as a Lyapunov function.  Indeed, by the standard existence and uniqueness theorem for stochastic differential equations, strong solutions of~\eqref{eqn:main} are defined and unique until the random time $\tau_{x}$ when the process started from $x\in \mathcal{X}$ exits every set $\mathcal{X}_n$, $n\in \N$, defined by 
\begin{align}
\label{eqn:Xn}
\mathcal{X}_n= \mathcal{O}_n \times B_n(0) \times (-n,n)
\end{align}
 where $B_n(0)$ denotes the open ball of radius $n$ in $(\R^k)^N$ centered at the origin.  Letting $\mathscr{L}$ denote the generator of the Markov process~\eqref{eqn:main}, one can then conclude that $\P\{ \tau_{x} = \infty\}=1$ for all $x \in \mathcal{X}$ after noting that on $\mathcal{X}$  
 \begin{align}
 \label{eqn:primitive} \mathscr{L} H(q,p, \xi) &= - \xi k_B T kN -\gamma \sum_{i=1}^N  \frac{|p_i|^2}{m_i^2}  + \sum_{i=1}^{N} \frac{\gamma k_B T }{m_i}\\
\nonumber &\leq \alpha H(q,p, \xi) + K
 \end{align}   
 for some constants $\alpha,K>0$.  That is, the conclusion that $\P\{ \tau_{x} = \infty\}=1$ for all $x \in \mathcal{X}$ then follows by a standard Gronwall comparison argument, for it implies that $t\mapsto \E_{x} H(x(t))$ grows no faster than an exponential.  See, for example, \cite{Her_11, Kh_12, MTIII, RB_06}.

\begin{remark}
For our purposes, the issue with equality~\eqref{eqn:primitive} is that it does not predict a convergence rate to the presumed equilibrium measure $\mu$.  This is because $\mathscr{L}H(q,p, \xi)$ is positive and large when $(q,p,\xi)\in \mathcal{X}$ is such that $\xi\ll -1$ and $|p|$ is bounded.  Intuitively, this means that dissipation in the system is not due to pointwise contraction of the Hamiltonian at large energies.  Rather, if one expects relaxation to equilibrium, it must be due to averaging effects in the system not captured by the (pointwise) equality~\eqref{eqn:primitive}.                 
 \end{remark}

By Lemma~\ref{lem:exist}, equation~\eqref{eqn:main} induces a Markov semigroup $(\mathscr{P}_t)_{t\geq 0}$ acting on bounded, $\mathcal{B}(\mathcal{X})$-measurable functions $\varphi: \mathcal{X}\rightarrow \R$ via   
\begin{align*}
\mathscr{P}_t \varphi(x)= \E_{x} \varphi(x(t))
\end{align*}
for all $t\geq 0$ and dually on probability measures $\nu$ on $\mathcal{B}(\mathcal{X})$ by 
\begin{align*}
(\nu \mathscr{P}_t)(A)= \int_\mathcal{X} \nu(dx) \mathscr{P}_t \textbf{1}_A(x)   
\end{align*}
for $t\geq 0$ and $A\in \mathcal{B}(\mathcal{X})$.  In the above, $\textbf{1}_A$ denotes the indicator function on the set $A$.  We shall use the notation 
\begin{align*}
\mathscr{P}_t(x, A)= \mathscr{P}_t \textbf{1}_A(x)
\end{align*}
for $t\geq 0$ and $A\in \mathcal{B}(\mathcal{X})$ to denote the transition probabilities of the Markov process $x(t)$.  

We call a probability measure $\nu$ on $\mathcal{B}(\mathcal{X})$ an \emph{invariant probability measure} for the Markov process $x(t)$ if $\nu \mathscr{P}_t= \nu$ for all $t\geq 0$.  Later, we will see that the augmented Gibbs measure $\mu$ in~\eqref{eqn:Gibbsdef} is the unique invariant probability measure for $x(t)$.  To measure convergence to $\mu$, for $W: \mathcal{X}\rightarrow (0, \infty)$ measurable we let $\mathcal{M}_W$ denote the set of probability measures $\nu$ on $\mathcal{B}(\mathcal{X})$ such that $W\in L^1(\nu)$.  We equip the set $\mathcal{M}_W$ with a metric $\rho_W: \mathcal{M}_W\times \mathcal{M}_W\rightarrow [0, \infty)$ given by 
\begin{align*}
\rho_W(\nu_1, \nu_2) =\sup_{\|\varphi\|_W \leq 1} \int_\mathcal{X} \varphi(x)(\nu_1(dx)-\nu_2(dx)) 
\end{align*}
where the weighted supremum norm $\| \varphi \|_W$ is defined for measurable $\varphi: \mathcal{X}\rightarrow \R$ by
\begin{align*}
\| \varphi\|_W = \sup_{x\in \mathcal{X}} \frac{|\varphi(x)|}{1+ W(x)}
\end{align*}

We now state the main result of the paper. 

\begin{theorem}
\label{thm:main}
We have the following:
\begin{itemize}
\item[(a)]  The augmented Gibbs measure  $\mu$ defined in~\eqref{eqn:Gibbsdef} is the unique invariant probability measure for the Markov process $x(t)$ solving~\eqref{eqn:main}.  
\item[(b)]  There exists $\beta_*>0$ such that for all $\epsilon, \beta_0$ with $0<\epsilon < \beta_0< \beta_*$ there exists $W\in C^2(\mathcal{X}; (0, \infty))$ and constants $C, \eta>0$ for which
\begin{align*}
\exp(( \beta_0 - \epsilon) H) \leq W\leq \exp( (\beta_0 + \epsilon) H)
\end{align*}     
is satisfied on $\mathcal{X}$ and such that the following bound holds for all $\nu \in \mathcal{M}_W$ and all $t\geq 0$
\begin{align*}
\rho_W(\nu \mathscr{P}_t, \mu) \leq C e^{-\eta t} \rho_W(\nu, \mu).  
\end{align*}
\end{itemize}
\end{theorem}

As we recall that $\xi$ is a fictitious control variable, consider now the marginal probability distributions $\bar{\mu}$ of $\mu$ and $\overline{\mathscr{P}}_t(x, \, \cdot\,)$ of $\mathscr{P}_t(x, \, \cdot \,)$, $x\in \mathcal{X}$, defined on the Borel sets $\mathcal{B}(\mathcal{Y})$ of $\mathcal{Y}:=\mathcal{O} \times (\R^k)^N$ by 
\begin{align*}
\bar{\mu} (A) &= \sqrt{\frac{2\pi}{\mathcal{Z}^2 a\beta}} \int_A e^{-\beta H(q,p,0)}\, dq \, dp \,\, \text{ and } \,\,\bar{\mathscr{P}}_t(x, A) &=\int_A  \int_{\R} \mathscr{P}_t(x, d\xi \, dq \, dp)    
\end{align*}
for $A\in \mathcal{B}(\mathcal{Y})$.  Note that $\bar{\mu}$ is the canonical Boltzmann-Gibbs measure
\begin{align}
\bar{\mu} (dq \, dp) =  \sqrt{\frac{2\pi}{\mathcal{Z}^2 a\beta}} \exp( - \beta( \|p\|_m^2/2 + U(q))) \, dq \, dp.  
\end{align}  
By combining Birkoff's ergodic theorem (see, for example, \cite{RB_06}) with Theorem~\ref{thm:main}, we immediately obtain the following corollary.   
\begin{corollary}
For every $f\in L^1(\bar{\mu})$:
\begin{align*}
\frac{1}{t}\int_0^t f(q(s), p(s)) \, ds \rightarrow  \int_\mathcal{Y} f(y) \, \bar{\mu}(dy)\,\, \text{ as } \,\, t\rightarrow \infty
\end{align*}  
where the convergence holds $\mu$-almost surely and in the $L^1(\mu)$-sense.  Furthermore, if $W\in C^2(\mathcal{X};(0, \infty))$ is any such function given in the conclusion of Theorem~\ref{thm:main} and $C, \eta>0$ are the associated constants in part \emph{(b)} of the result, then for all $t\geq 0$ and all $x \in \mathcal{X}$ we have
\begin{align}
\label{eqn:marginalbound}
\|\overline{\mathscr{P}}_t(x, \, \cdot\,) - \bar{\mu} \|_{TV} \leq  C e^{-\eta t} \sup_{\| \phi \|_W\leq 1} \bigg\{\phi(x)- \int_\mathcal{X} \phi(x') \mu (dx')  \bigg\} 
\end{align}   
where $\| \cdot \|_{TV}$ denotes the total variation distance.  
\end{corollary}

Before proceeding further, we make some remarks.  

\begin{remark}
\label{rem:dawson}
In addition to proving Theorem~\ref{thm:main}, we also provide an estimate on the parameter $\beta_*>0$ in the statement of the result which depends on the maximum, which we will denote by $D_\text{max}$, of \emph{Dawson's integral}.  Dawson's integral is a special function $D:\R\rightarrow \R$ defined by      
\begin{align}
D(z) = \exp(-z^2)\int_0^z \exp(y^2)\, dy
\end{align}
and arises in heat conduction and the theory of electrical oscillators~\cite{harris_48, hummer_62, lether_97, mccabe_74, sajo_93}.  From this point of view, perhaps it is not surprising that it would turn up in the setting of equation~\eqref{eqn:main}.  To understand why it is surprising from another perspective, later we will see that we can pick
\begin{align}
\label{eqn:betastar}
\beta_*=\frac{1}{8D_\text{max}^2}  \beta =  0.42701... \times \frac{1}{k_B T}  
\end{align}
and that, to the best of the author's knowledge, the argument given here does not allow for larger $\beta_*$. 
From working on other stochastic Hamiltonian dynamics, however, one expects to be able to choose $\beta_*= \beta=1/(k_BT)$ due to the asymptotic behavior of the probability density of the augmented Gibbs measure  $\mu$ as $H\rightarrow \infty$.  Thus from this perspective, there appears to be a discrepancy in the arguments of 
\begin{align*}
\beta-\frac{1}{8D_\text{max}^2}  \beta. 
\end{align*}   
The question hence becomes: Is $\beta_*$ given in~\eqref{eqn:betastar} optimal and if so, why is this the case?  In the theoretical arguments that identify the threshold, the behavior of the fictitious control variable $\xi$ in the region where $\xi \ll -1$ and $p$ is bounded plays a fundamental role.  Recall from~\eqref{eqn:primitive} that this is the region where the energy of the system is increasing pointwise.  We will see that $D$, hence $D_\text{max}$, is related to the exit distribution of the process from this ``bad" part of space, so one might conjecture that $D_\text{max}$ is a fundamental parameter governing the stability of the system.  Thus the given threshold may not be so surprising.                     
\end{remark}

\begin{remark}
One can improve the lefthand side of the inequality~\eqref{eqn:marginalbound} by replacing it with the weighted total variation distance
\begin{align}
\sup_{\substack{\phi:\mathcal{Y}\rightarrow \R\\\| \phi\|_{W} \leq 1}} \int_{\mathcal{Y}} \phi(y) (\overline{\mathscr{P}}_t(x, dy) -  \bar{\mu}(dy))
\end{align}
where the supremum is taken over all measurable $\phi: \mathcal{Y}\rightarrow \R$ satisfying $$\sup_{(q,p, \xi) \in \mathcal{X}} \frac{|\phi(q,p)|}{1+W(q,p, \xi)}\leq 1.$$  Also note that the righthand side of~\eqref{eqn:marginalbound} can be made more explicit by bounding it above by
\begin{align}
C e^{-\eta t} \bigg\{ 2 +W(x) + \int_\mathcal{X} W(x') \mu(dx') \bigg\}. 
\end{align} 
\end{remark}

The proof of Theorem~\ref{thm:main} splits into two parts: the existence of the appropriate Lyapunov function $W$ to ensure regular and sufficiently fast return times to a compact set in $\mathcal{X}$, as stated in Theorem~\ref{thm:lyapexist} below, and the necessary support and regularity properties of the Markov transitions outlined in Proposition~\ref{prop:suppreg}.      
\begin{theorem}
\label{thm:lyapexist}
Recall that $\mathscr{L}$ denotes the generator of the Markov process $x(t)$.  Let $\epsilon >0$ and fix $\beta_0 >0$ satisfying
\begin{align}
\label{eqn:beta0bound}
\beta_0< \beta_*:=\frac{\beta}{8 D_\text{\emph{max}}^2} .  
\end{align}
Then there exists $W\in C^2 (\mathcal{X})$ and constants $ \alpha, K>0$ such that the following two estimates hold on $\mathcal{X}$
\begin{align}
\tag{I}
\exp((\beta_0 -\epsilon) H) \leq W \leq \exp( (\beta_0 +\epsilon)H),  
\end{align}
\begin{align}
\tag{II}
\mathscr{L} W \leq - \alpha W + K.    
\end{align}
\end{theorem}
The proof of Theorem~\ref{thm:lyapexist} is given in Section~\ref{sec:lyap} and is motivated by the heuristics in Section~\ref{sec:heuristics}.

Fundamental to the supports of the Markov transitions is the notion of the \emph{arc length} $L_\gamma$ of a curve $\gamma \in C^1([0,1]; \mathcal{O})$ with respect to the norm $\| \cdot \|_m$ defined by   
\begin{align*}
L_\gamma = \int_0^1 \|\dot{\gamma}(t)\|_m \, dt.  
\end{align*}  
Since $\mathcal{O}$ is a connected subset of Euclidean space, $\mathcal{O}$ is path connected.  Hence this notion of arc length allows us to define the $\mathcal{O}$-\emph{distance} $L_{q,q'}$ between $q, q'\in \mathcal{O}$ by  
\begin{align*}
L_{q,q'} = \inf \{L_\gamma\, : \, \gamma \in C^1([0,1]; \mathcal{O}), \, \gamma(0)=q, \gamma(1)=q'\}.  
\end{align*} 
Clearly if $\mathcal{O}=(\R^k)^N$, then $L_{q,q'}= \|q-q'\|_m$.  However, if $\mathcal{O}\subsetneq (\R^k)^N$ as in the case of a potential with singularities, then the shortest distance between $q, q' \in \mathcal{O}$ in the norm $\|\cdot \|_m$ is $L_{q,q'}$.  For $x=(q, p, \xi) \in \mathcal{X}$ and $t>0$, define 
\begin{align}
\label{eqn:suppset}
\mathcal{A}(x, t)= \big\{ (q', p', \xi') \in \mathcal{X} \,  : \,  \xi' \geq \xi  + (ta)^{-1} L^2_{q,q'}- t a^{-1} k_B TkN\big\}.  
\end{align}
We have the following.
\begin{proposition}
~~~
\label{prop:suppreg}
\begin{itemize}
\item[(i)]  For each $x\in \mathcal{X}$ and $t>0$, the measure $\mathscr{P}_t(x, \, \cdot \,)$ is absolutely continuous with respect to Lebesgue measure on $\mathcal{X}$.  Denoting the probability density of $\mathscr{P}_t(x, \, \cdot\,)$ by $r_t(x,y)$, the mapping $(t,x,y) \mapsto r_t(x,y) : (0, \infty) \times \mathcal{X} \times \mathcal{X}\rightarrow [0, \infty)$ is continuous.
\item[(ii)]  For any $x\in \mathcal{X}$ and $t>0$
\begin{align*}
\supp \mathscr{P}_t(x, \, \cdot \,) = \mathcal{A}(x,t)
\end{align*}
where $\supp \nu$ denotes the support of the measure $\nu$ and $\mathcal{A}(x,t)$ is as in~\eqref{eqn:suppset}.    
\end{itemize}  
\end{proposition}

We will see in the Appendix that Theorem~\ref{thm:lyapexist} and Proposition~\ref{prop:suppreg} together imply Theorem~\ref{thm:main}.  This implication follows from Theorem 1.2 of \cite{HM_08}.  In essence, one has to translate the conclusions of Theorem~\ref{thm:lyapexist} and Proposition~\ref{prop:suppreg} to, respectively, Assumption~1 and Assumption~2 of~\cite{HM_08} to an embedded Markov chain.  Thus Theorem 1 of~\cite{HM_08} can be applied to that chain and then one has to translate the bound obtained back to one for the original process, hence producing Theorem~\ref{thm:main}.  Also, while not entirely nontrivial, the proof of Proposition~\ref{prop:suppreg} relies on by now standard methods to produce the necessary smoothing of the Markov semigroup via H\"{o}rmander's theorem~\cite{Hor_67} and support properties of the associated transitions via the support theorems~\cite{SV_72, SV_73}.  This proof will be given in Section~\ref{sec:sup_con}.  The main innovative content of the document is the proof of Theorem~\ref{thm:lyapexist}.

\section{Hueristics}
\label{sec:heuristics}
As emphasized at the end of the previous section, the main difficulty in proving Theorem~\ref{thm:main} is the absence of a natural Lyapunov function.  This is clearly evident in equation~\eqref{eqn:primitive} which shows that $\mathscr{L} H$ is positive and large in the region where $\xi \ll-1$ and $|p|$ is bounded.  Thus if we expect dissipation in the system at large energies, then it must be due to averaging effects not captured by the pointwise equality~\eqref{eqn:primitive}.  As in the works~\cite{cooke17:_geomet_ergod_two, HerMat17}, we will look for a small perturbation $\psi$ of the Hamiltonian $H$ that encapsulates these effects.  The construction of such a perturbation uses a slight modification of the procedure developed in~\cite{AKM12, HerMat15}.      

\subsection{A simplifying ansatz}
Guided by the behavior of the density of $\mu$ in~\eqref{eqn:Gibbsdef} as $H\rightarrow \infty$, fixing $\epsilon >0$ and $\beta_0 >0$ small enough we look for a function $\psi \in C^2(\mathcal{X})$ such that 
\begin{align}
\label{eqn:Wform}
W(x)= \exp( \beta_0 H(x) + \psi(x))
\end{align}  
satisfies conclusions (I) and (II) of Theorem~\ref{thm:lyapexist}.  First note that if $V= \beta_0 H + \psi$, then applying the generator $\mathscr{L}$ to $W$ of the form~\eqref{eqn:Wform} produces   
\begin{align*}
\mathscr{L}W &= W\{  \mathscr{L}V + \gamma k_B T |\nabla_p V|^2\}.     
\end{align*}  
Thus for this choice of $W$, conclusion (II) of Theorem~\ref{thm:lyapexist} holds if and only if the following bound holds
\begin{align}
\label{eqn:Vbound1}
\mathscr{L} V + \gamma k_B T |\nabla_p V|^2 \leq - \alpha + K \textbf{1}_A
\end{align} 
on $\mathcal{X}$ for some $\alpha, K>0$ and $A\subseteq \mathcal{X}$ compact.  Note clearly that if the property given in~\eqref{eqn:Vbound1} holds, then 
\begin{align}
\label{eqn:Vbound2}
\mathscr{L}V \leq - \alpha + K \textbf{1}_A
\end{align} 
holds on $\mathcal{X}$ for the same choice of $\alpha, K >0$ and $A\subseteq \mathcal{X}$ compact.  What is surprising is that the reverse implication is often true as well.  In other words, by slightly tweaking $V$ satisfying~\eqref{eqn:Vbound2} we can also obtain the stronger property~\eqref{eqn:Vbound1}, but perhaps for a different choice of $\alpha, K, A$.  

The reverse implication can be intuited by scaling $V$ by a small constant $\epsilon>0$; that is, if one sets $V_\epsilon= \epsilon V$, then the gradient term in~\eqref{eqn:Vbound1} is of order $\epsilon^2$ while the $\mathscr{L}V_\epsilon$ is of order $\epsilon$.  This also seems plausible when one considers the case when $\psi=0$, so $V= \beta_0 H$ and hence
\begin{align*}
\mathscr{L} V + \gamma k_B T |\nabla_p V|^2 &= \beta_0 \mathscr{L}H + \beta_0^2 \gamma k_B T |\nabla_p H|^2 \\
&= - \gamma\beta_0 (1-\beta_0/\beta) \sum_{i=1}^N \frac{|p_i|^2}{m_i^2} - \beta_0 \xi k_B TkN + \beta_0 \sum_{i=1}^N \frac{\gamma k_B T}{m_i} .  
\end{align*}      
Note that as $\beta_0 \searrow 0$, after consulting~\eqref{eqn:primitive} we see that the expression on the righthand side above is asymptotically equal to $\beta_0 \mathscr{L}V$.  

We these simple observations in mind, we are now squarely concerned with finding a perturbation $\psi$ so that $V=\beta_0 H + \psi$ satisfies the presumably weaker condition~\eqref{eqn:Vbound2}.

\subsection{Subsolutions of $\mathscr{L}V= -\alpha$ for $H\gg 1$}
Employing our ansatz, fixing $\alpha >0$ the goal is to now find subsolutions of the PDE
\begin{align}
\label{eqn:LVa}
\mathscr{L}V = - \alpha \,\, \text{ on } \,\, H \geq R
\end{align}  
of the form $V= \beta_0 H + \psi$ where $\psi$ is the unknown and $R>0$ is large.  Of course, we also want to be able to ``tune" the perturbation $\psi$ so that for any $\epsilon >0$ we can construct $\psi=\psi_\epsilon$ to be a subsolution of~\eqref{eqn:LVa} with the additional property that $|\psi| \leq \epsilon H$ on $\mathcal{X}$.  As we will see below, however, this additional property is a simple consequence of the structure of subsolutions.     

We first consult equation~\eqref{eqn:primitive} to see that in the region 
\begin{align*}
 \mathcal{R}_0= \big\{ (q,p, \xi) \in \mathcal{X} \, : \, \xi \geq K_*\,\, \text{ or }\,\, |p|^2 \geq p_* \sqrt{\xi^2 +1}\big\}
 \end{align*}
for $K_*, p_*>0$ sufficiently large we have that 
\begin{align*}
\mathscr{L}(\beta_0 H) = \beta_0 \mathscr{L} H \leq -\alpha.      
\end{align*}    
Thus we need not perturb off of the Hamiltonian to have the desired effect in this region.  In other words, we should set 
\begin{align}
\psi=0 \,\, \text{ on } \,\, \mathcal{R}_0
\end{align}
and restrict our analysis of the problem~\eqref{eqn:LVa} to the complement
\begin{align*}
\mathcal{R}_0^c= \big\{ (q,p, \xi) \in \mathcal{X}\, : \, \xi \leq K_*\,\, \text{ and } \,\, |p|^2 \leq p_* \sqrt{\xi^2 +1}\big\}.  
\end{align*}

To help reduce the difficulty in finding subsolutions of~\eqref{eqn:LVa} on $\mathcal{R}_0^c$, similar to the works~\cite{cooke17:_geomet_ergod_two, HerMat17} we will study the dynamics at large energies.  This will be done by using formal scaling analysis applied to the infinitesimal generator 
\begin{align}
\label{eqn:generator}
\mathscr{L} &= \sum_{i=1}^N \frac{p_i}{m_i}\cdot \nabla_{q_i} - \sum_{i=1}^N \big( \xi + \frac{\gamma}{m_i}\big) p_i \cdot \nabla_{p_i}   - \nabla U \cdot \nabla_p \\
\nonumber &\qquad + a^{-1}(\|p\|_m^2 - kN/\beta) \partial_\xi   + \frac{\gamma}{\beta} \Delta_p.  
\end{align}    
This analysis allows us to at least heuristically justify neglecting terms in~\eqref{eqn:generator} when solving~\eqref{eqn:LVa} on $\mathcal{R}_0^c$.  For simplicity, we restrict our discussion in this section to the case when $N=k=m_1=a=k_BT=1$, so that 
\begin{align}
\label{eqn:Lrestr}
\mathscr{L}= p \partial_q - (\xi + \gamma) p \partial_p - U' \partial_p + (p^2 - 1) \partial_\xi  + \gamma  \partial_p^2.   
\end{align}  
Furthermore, to make the scaling analysis that follows more explicit, we will assume that our potential $U: \R\rightarrow [0, \infty)$ is a polynomial of degree $\ell\geq 2$.  Note that such a potential satisfies the hypotheses made at the beginning of Section~\ref{sec:notation}.     
  
We start by considering the subregion $\mathcal{R}_1'$ in $\mathcal{R}_0^c$ where $\xi$ is assumed to be bounded.  In particular, $|p|$ must also be bounded in $\mathcal{R}_1'$.  Consequently, as $H\rightarrow \infty$ in $\mathcal{R}_1'$, $U\rightarrow \infty$ while $\xi$ and $|p|$ remain bounded.  Thus making the substitution $q=\lambda Q$ for $\lambda \gg1$ and not changing $\xi$ and $p$ in the expression~\eqref{eqn:Lrestr} produces 
\begin{align*}
\mathscr{L} = -\lambda^{\ell-1} U'_L(Q) \partial_p + r(Q, p, \xi, \lambda) 
\end{align*}  
where $U_L'$ is the leading order term in $U'$ and the remainder term $r$ is $o(\lambda^{\ell-1})$ as $\lambda \rightarrow \infty$.  Since $\ell \geq 2$, this argument suggests that in the region $\mathcal{R}_0^c$ where $\xi$ is bounded
\begin{align}\label{eqn:approx1}
\mathscr{L} \approx - U'(q) \partial_p.  
\end{align}  
A nearly identical scaling argument yields the same approximation~\eqref{eqn:approx1} in the region
\begin{align}
\label{eqn:R1def}
\mathcal{R}_1 = \mathcal{R}_0^c \cap \big\{ (q,p, \xi) \in \mathcal{X} \, : \, |U'(q)| \geq U_* ( \xi^6 + 1)^{1/4}\big\}
\end{align} 
for $U_* >0 $ large enough.  Note that the set $\mathcal{R}_1$ subsumes the region in $\mathcal{R}_1'$ where $H$ is large.  Also, the power $(\xi^6)^{1/4}= |\xi|$ in~\eqref{eqn:R1def} above balances the term $-(\xi  + \gamma) p$ on the boundary of $\mathcal{R}_0^c$ where $|p|^2 = p_* \sqrt{\xi^2 + 1}$, so when $U_*>0$ is large enough, the $-U' \partial_p$ term still dominates.  

Translating this scaling analysis back to solving the problem~\eqref{eqn:LVa} in the region $\mathcal{R}_1$, we consider solving the equation
\begin{align}
\label{eqn:approx1psi}
- U' \partial_p \psi =  - \alpha_1 \sqrt{\xi^2+ 1}  
\end{align}  
for some constant $\alpha_1 >0$.  Note that the righthand side of~\eqref{eqn:approx1psi} is natural since in $\mathcal{R}_0^c$
\begin{align}
\label{eqn:Lbehaviorpsmall}
|\mathscr{L} (\beta_0H)|\leq C\sqrt{\xi^2 +1}
\end{align}
for some $C>0$.  Thus for $\alpha_1 >C$, the solution $\psi$ is designed to counteract any ``bad" parts in $\beta_0\mathscr{L} H$ arising in the region $\mathcal{R}_1$.   Observe that in $\mathcal{R}_1$, equation~\eqref{eqn:approx1psi} clearly has a particular solution defined by 
\begin{align}
\label{eqn:approx1psi1}
\psi(q,p, \xi) = \frac{\alpha_1 p \sqrt{\xi^2 + 1}}{U'},  
\end{align}
thus giving a natural choice for the perturbation $\psi$ in $\mathcal{R}_1$.  Note that this choice of $\psi$ can be made arbitrarily small on $\mathcal{R}_1$ by choosing $U_*>0$ large enough.

We next turn our attention to the region 
\begin{align}
\mathcal{R}_2 = \mathcal{R}_0^c \cap \mathcal{R}_1^c = \big\{ (q, p, \xi) \in \mathcal{X} \, : \, \xi \leq K_*, \, |p|^2 \leq p_*\sqrt{\xi^2 +1}, \, |U'| \leq U_*(\xi^6 + 1)^{1/4}\big\}.    
\end{align}     
In $\mathcal{R}_2$, any route to infinite energy must have $\xi\rightarrow -\infty$.  It is with this fact and how the boundaries in $\mathcal{R}_2$ for $|p|^2$ and $|U'|$ scale in $\xi$ that we introduce the following scaling substitutions:
\begin{align}
\xi= \lambda \Xi, \,\,\, \,p = (c_1 \lambda^{1/2} + c_2 + c_3\lambda^{-\delta}) P, \,\,\,\, q= (c_4\lambda^{\frac{3}{2(\ell-1)}}+ c_5)Q
\end{align}
where $(Q, P, \Xi)$ are the new variables, and $c_i \in \R$, $\delta>0$ are constants.  When $c_1=c_2=0$, the parameter $\delta>0$ allows us to analyze the dynamics when $|p|$ is small.  Now observe that after making this substitution, the generator becomes
\begin{align*}
\mathscr{L} &= \frac{ c_1 \lambda^{1/2} + c_2 + c_3\lambda^{-\delta}}{c_4 \lambda^{\frac{3}{2(\ell-1)}}+ c_5} P \partial_Q - (\lambda \Xi + \gamma) P \partial_P - \frac{U'((c_4 \lambda^{\frac{3}{2(\ell-1)}}+ c_5) Q)}{c_1 \lambda^{1/2} + c_2 + c_3\lambda^{-\delta}} \partial_P \\
&\qquad + \lambda^{-1}\big((c_1 \lambda^{1/2} + c_2 + c_3\lambda^{-\delta})^2-1 \big) \partial_\Xi + \gamma \frac{1}{(c_1 \lambda^{1/2} + c_2 + c_3\lambda^{-\delta})^2} \partial_P^2.  
\end{align*}   
Hence if $c_1=c_4=0$ and $c_2, c_5 \neq 0$, or if $c_1\neq 0, c_4=0$ and $c_5\neq 0$, we see that as $\lambda \rightarrow \infty$
\begin{align*}
\mathscr{L} \approx - \lambda \Xi P \partial_P.
\end{align*}  
Next, if $c_1=0, c_4\neq 0$ and $c_2\neq 0$ we have that as $\lambda \rightarrow \infty$ 
\begin{align*}
\mathscr{L} \approx - \lambda^{3/2} U'_L(c_4 Q) \partial_P
\end{align*}
where we recall that $U'_L$ is the leading order term in the polynomial $U'$.  Also, if $c_1, c_4 \neq 0$ we have that as $\lambda \rightarrow \infty$
\begin{align*}
\mathscr{L} \approx - \lambda \Xi P \partial_P - \lambda \frac{U'_L(c_4)}{c_1} \partial_P
\end{align*}
where we are neglect any cancellation that could occur between these two leading-order terms.  
Finally, if $c_1=c_2=0$ and $c_3\neq 0$, it is evident that as $\lambda \rightarrow \infty$ 
\begin{align*}
\mathscr{L} \approx  - \lambda \Xi P \partial_P - \frac{\lambda^\delta}{c_3} U'(c_4 \lambda^{\frac{3}{2(\ell-1)}} + c_5) \partial_P + \frac{\gamma}{c_3^2} \lambda^{2\delta} \partial_P^2.   
\end{align*}
Thus in the region~$\mathcal{R}_2$, the scaling analysis gives that at large energies 
\begin{align*}
\mathscr{L} \approx \mathscr{A}:= -(\xi + \gamma) p \partial_p - U'(q) \partial_p + \gamma \partial_p^2.  
\end{align*}

\begin{remark}
Of course we could break apart the region in $\mathcal{R}_2$ further according to how $\mathscr{L}$ changes above, but the operator $\mathscr{A}$ is simple enough to work with and encapsulates the dominant behaviors identified in the scaling analysis.  Note that the reason $\mathscr{A}$ is simple is that the dynamics driven by this operator is constant in $\xi$ and $q$.  In other words, the scalings suggest that in this region at large energies, the
Markov process associated to $\mathscr{L}$ is approximated well by the Ornstein-Uhlenbeck dynamics
\begin{align}
\label{eqn:OU}
d P_t = -(\xi + \gamma) P_t \,dt - U'(q) \, dt + \sqrt{2\gamma} \, dB_t
\end{align}
where $B_t$ is a standard, real-valued Brownian motion and $\xi\ll -1$ and $q\in \R$ are fixed constants.  \end{remark}

Motivated by this analysis and relation~\eqref{eqn:Lbehaviorpsmall}, for $H\gg1$ in $\mathcal{R}_2$ we should take our perturbation $\psi$ to satisfy
\begin{align}
\label{eqn:pois2}
\mathscr{A} \psi = - \alpha_2 |\xi +\gamma|
\end{align} 
for some constant $\alpha_2 >0$.  Note that in this case we do not need a large negative constant on the righthand side of the equation~\eqref{eqn:pois2} because $|\xi|\rightarrow \infty$ at large energies in $\mathcal{R}_2$.  Also, the appearance of $\gamma$ in the formula~\eqref{eqn:pois2} does not change the qualitative behavior of the solution but it makes the explicit formula below more compact.  Specifically, a particular solution of equation~\eqref{eqn:pois2} is given by 
\begin{align}
\label{eqn:psipert2}
\psi(q,p, \xi) = - 2\alpha_2 \int_0^{\frac{|\xi+ \gamma|^{1/2}}{\sqrt{2\gamma}} \big(p - \frac{U'(q)}{|\xi + \gamma|} \big)} D(z) \, dz
\end{align} 
where $D:\R\rightarrow \R$ is Dawson's integral, which was introduced and discussed in Remark~\ref{rem:dawson}.  By using some well-known asymptotics for the function $D$, we will see in the next section that for any $\epsilon >0$, $|\psi(q,p, \xi)|\leq \epsilon H(q,p, \xi)$ in $\mathcal{R}_2$ for all $|\xi| >0$ large enough depending on $\epsilon$.  Thus using the appropriate cutoff functions, we will see that we can construct, for a given $\epsilon >0$, our desired subsolution $\psi$ with $|\psi|\leq \epsilon H$ globally.

\section{The Lyapunov Function}
\label{sec:lyap}

Following the ideas of Section~\ref{sec:heuristics}, in this section we prove Theorem~\ref{thm:lyapexist}.  We recall that the the Lyapunov function $W$ in the statement of Theorem~\ref{thm:lyapexist} will be of the form~\eqref{eqn:Wform} where $H$ is the Hamiltonian~\eqref{eqn:Hdef} and $\psi \in C^2(\mathcal{X})$ is an appropriately chosen perturbation.  The heuristic scaling analysis coupled with the behavior of $\mathscr{L}H(q,p, \xi)$ when $\xi\ll-1$ and $|p|$ is bounded suggested two qualitatively different forms for $\psi$ in two different regions in $\mathcal{X}$.  Refer to expressions~\eqref{eqn:approx1psi1} and~\eqref{eqn:psipert2} for these two forms in the simplified case when $N=k=m_1=a=k_BT=1$.  These forms will be slightly generalized below to account for changes in dimensionality and parameters.  We also must cutoff each function when the asymptotic analysis is no longer valid.  This will then produce two globally-defined perturbations, which we denote by $\psi_1$ and $\psi_2$ below.  

\begin{remark}
\label{rem:psi0}
We will need one more perturbation, denoted by $\psi_0$ below, which we did not motivate in the previous section.  The function $\psi_0$, however, should be thought of as an auxiliary perturbation which places slightly more weight in the Hamiltonian $H$ on the variable $\xi$ in the region where $\xi<0$.  While a small perturbation itself, it has the advantage inducing dissipation whenever the kinetic energy is large enough while not changing the essential behavior of $\mathscr{L}(\beta_0 H)$ in the ``bad" part of space where $\xi \ll -1$ and $|p|$ is bounded.  The function $\psi_0$ moreover is convenient in that it helps subsume various remainder terms brought on by $\psi_1$ and $\psi_2$.                              
\end{remark}

Following these remarks, Theorem~\ref{thm:lyapexist} is an immediate consequence of the following.
\begin{theorem}
\label{thm:maintech}
Fix $\beta_0>0$ satisfying~\eqref{eqn:beta0bound} and $\epsilon\in (0, \beta_0)$.  Then there exist $\psi_i \in C^2(\mathcal{X})$, $i=0,1,2$, satisfying the following two properties:
\begin{itemize}
\item[(i)]  $|\psi_0 + \psi_1 + \psi_2| \leq \epsilon H$. 
\item[(ii)]  If $V= \beta_0 H + \psi_0+ \psi_1 + \psi_2$, then there exists a compact set $A\subseteq \mathcal{X}$ and constants $\alpha, K>0$ for which the bound~\eqref{eqn:Vbound1} holds.  
\end{itemize} 
\end{theorem}

The proof of Theorem~\ref{thm:maintech} will be broken up into several smaller pieces.  In particular, as the functions $\psi_i$ are introduced, we will also deduce a series of estimates which, when combined at the end of the section, will imply Theorem~\ref{thm:maintech}.

\subsection{Perturbation $\psi_0$}
Our first perturbation, $\psi_0$, is the simplest.  In order to define it, let $f_0\in C^\infty(\R;[0,1])$ be a cutoff function satisfying the following conditions  
\begin{align*}
f_0(y) = \begin{cases}
1 & \text{ if } \,\, y \leq -1\\
0 & \text{ if }\,\, y \geq 0
\end{cases},   \qquad f'_0 \leq 0, \qquad \text{ and } \qquad |f'_0| \leq 2. 
\end{align*}
Let $\delta >0$ and for $(q, p, \xi) \in \mathcal{X}$ define
\begin{align}
\label{eqn:ham}
\psi_0(q,p, \xi) =   \delta f_0(\xi) \frac{a\xi^2}{2}.  
\end{align}
Then it is not hard to check that on $\mathcal{X}$
\begin{align}
\label{eqn:psi0bound0}
|\psi_0| \leq \delta H,
\end{align}
\begin{align}
\label{eqn:psi0bound1}
\mathscr{L}\psi_0(x,v, \xi) 
& \leq - f_0 \delta  |\xi| \|p\|_m^2  + f_0 \frac{\delta}{\beta} kN |\xi|  + \frac{\delta}{\beta} k N  
\end{align}
and 
\begin{align}
\label{eqn:psi0bound2}
\nabla_p \psi_0 = 0.   
\end{align}

\begin{remark}
As discussed in Remark~\ref{rem:psi0}, note that when $\delta >0$ in~\eqref{eqn:psi0bound1} is small, $\psi_0$ allows for a dissipative effect in the region where $\xi\ll -1$ and $\|p\|_m^2$ is bounded below by a sufficiently large positive constant.  Moreover, for $\delta >0$ small, the perturbation is small relative to $\beta_0 H$ and the tradeoff for introducing it is also small relative to $\mathscr{L} (\beta_0 H)$ in the sense that $\mathscr{L} \psi_0 \leq C \delta( |\xi| + 1)$ for some constant $C>0$ independent of $\delta$.      
\end{remark}

\subsection{Perturbation $\psi_1$}
Now to define $\psi_1$, let $K_*>0$ be a parameter and $f_i \in C^\infty(\R;[0,1])$, $i=1,2,3$, be cutoff functions satisfying 
\begin{align*}
f_1(y)= \begin{cases}
1 & \text{ if } y\leq  K_* \\
0 & \text{ if } y\geq  K_*+1
\end{cases}, \,\,\, 
f_2(y)= \begin{cases}
1 & \text{ if } |y|\leq  1  \\
0 & \text{ if } |y|\geq 2
\end{cases},\,\,\, 
f_3(y)= \begin{cases}
1 & \text{ if } |y|\geq 2 \\
0 & \text{ if } |y|\leq  1
\end{cases} .
\end{align*}
 Let $p_*, U_*>0$ be parameters and set 
\begin{align*}
g_1(q,p, \xi) = f_1(\xi) f_2 \bigg( \frac{|p|^2}{p_*\sqrt{\xi^2 + 1}}\bigg) f_3 \bigg( \frac{|\nabla U(q)|^2}{U_* (\xi^2+1)} \bigg)
\end{align*}
Let $\alpha_1>0$ and consider $\psi_1$ defined by 
\begin{align}
\label{eqn:psi1def}
\psi_1(q,p, \xi) =\begin{cases}
g_1(q, p, \xi) \alpha_1\sqrt{\xi^2  + 1} \displaystyle{\frac{ p\cdot \nabla U(q)}{|\nabla U(q)|^2}} & \text{ if } |\nabla U(q)|^2 \geq U_*/2\\
 0 & \text{ otherwise}  
\end{cases}.\end{align} 
 Observe that by construction $\psi_1 \in C^\infty(\mathcal{X})$.  We will now prove the following.
 \begin{lemma}
 \label{lem:psi1}
For any $\epsilon, \alpha_1, K_* >0$, by first picking $p_*>0$ large enough and then picking $U_*>0$ large enough we have the global bounds on $\mathcal{X}$
\begin{align}
\label{eqn:psi1est0}
|\psi_1| \leq \epsilon H, 
\end{align}
 \begin{align}
 \label{eqn:psi1est1}
 \mathscr{L} \psi_1 &\leq- g_1 \alpha_1 \sqrt{\xi^2 + 1} +  \epsilon f_0|\xi|  \|p\|_m^2 + \epsilon |p|^2 + \epsilon f_0 |\xi| + \epsilon,  
 \end{align}
 \begin{align}
\label{eqn:psi1est2} |\nabla_p \psi_1| &\leq \epsilon.  
 \end{align} 
\end{lemma}

\begin{remark}
Note that the region where $g_1 \neq 0$ does not quite coincide with the form of the region $\mathcal{R}_1$ introduced in~\eqref{eqn:R1def}.  In particular, by considering powers of $\xi$, the region here is larger than $\mathcal{R}_1$.  The fact that $\psi_1$ provides the needed Lyapunov estimate on this larger region is only made possible by the presence of $\psi_0$, as it allows us to estimate meddling remainder terms, e.g. $\epsilon |\xi| \|p\|_m^2$ in~\eqref{eqn:psi1est1}, for which $\psi_1$ itself cannot account.  Noticing this fact was crucial in the analysis because it affords the luxury of working with normal, as opposed to a more restricted class of, potentials $U$.                  
\end{remark}

\begin{proof}[Proof of Lemma~\ref{lem:psi1}]
The first estimate~\eqref{eqn:psi1est0} follows easily after noting that 
\begin{align*}
|\psi_1(q,p, \xi)| \leq \frac{2\alpha_1 p_*^{1/2}}{U_*^{1/2}} (\xi^2 +1)^{1/4}.  
\end{align*}
We now turn to the issue of estimating $\mathscr{L} \psi_1$, which we split into three parts as follows
\begin{align*}
\mathscr{L} \psi_1 = \mathscr{T}_1 \psi_1 + \mathscr{A}\psi_1 + \mathscr{T}_2 \psi_1
\end{align*}
where $\mathscr{T}_1 = \sum_{i=1}^N m^{-1}_i p_i\cdot \nabla_{q_i}$,  $\mathscr{T}_2 = a^{-1}( \|p\|_m^2 - kN/\beta) \partial_\xi$ and
\begin{align}
\label{eqn:opdef}
 \mathscr{A} = \frac{\gamma}{\beta}\Delta_p - \nabla_{q} U \cdot \nabla_p+  \sum_{i=1}^N - \big(\xi + \frac{\gamma}{m_i}\big) p_i \cdot \nabla_{p_i}  \end{align}
 where we again recall that $\beta = 1/(k_B T)$.  
First note that if $c= \min_i m_i$, then 
 \begin{align*}
 \mathscr{T}_1 \psi_1 &= \sum_{i, \ell} g_1 \alpha_1 \sqrt{ \xi^2 + 1} m_{i}^{-1} p_i^\ell p\cdot \partial_{q_i^\ell}\bigg( \frac{\nabla U}{|\nabla U|^2} \bigg)\\
 &\qquad + \sum_{i, \ell, j,\ell'}  f_1 f_2 f_3' \alpha_1 \sqrt{\xi^2 + 1} m_i^{-1} p_i^\ell \frac{p\cdot \nabla U}{|\nabla U| ^2}\frac{2  (\partial_{q_j^{\ell'}} U)  \partial_{q_i^\ell q_{j}^{\ell'}}^2 U }{U_*( \xi^2 + 1)}\\
 &\leq   g_1 c^{-1} \alpha_1\sqrt{ \xi^2 + 1} |p|^2 |\nabla G| + 4 c^{-1} f_1 f_2 |f_3'| \alpha_1 \sqrt{ \xi^2 + 1}  |p|^2 \frac{|\nabla^2 U|}{|\nabla U|^2} 
 \end{align*}  
 where $G = \nabla U/|\nabla U|^2$.  Note that for every $\epsilon, \alpha_1,  K_* >0$ since $U$ is an normal potential (see Definition~\ref{def:normpot}) we may choose $U_*>0$ large enough to control the $\nabla G$ and $|\nabla^2 U|/|\nabla U|^2$ terms above and arrive at the global estimate   
\begin{align}
\label{psi1:b1}
\mathscr{T}_1 \psi_1(q,p, \xi) &\leq \epsilon f_0 |\xi| \|p\|_m^2 +\epsilon |p|^2 . 
 \end{align}
 Turning to the next term $\mathscr{A} \psi_1$, observe that 
   \begin{align*}
 \mathscr{A} \psi_1 
 &=\frac{4\gamma \alpha_1 + 2\gamma \alpha_1 k N}{\beta p_*}  f_1 f_2' f_3 \frac{p\cdot \nabla U}{|\nabla U|^2} + \frac{4\gamma \alpha_1}{\beta p_*} f_1 f_2'' f_3 \frac{|p|^2}{p_* \sqrt{\xi^2 + 1}} \frac{p \cdot \nabla U}{|\nabla U|^2}\\
 &\qquad - g_1 \alpha_1 \sqrt{\xi^2 + 1} + \frac{2\alpha_1}{p_*} f_1 f_2' f_3 \frac{(p\cdot \nabla U)^2}{|\nabla U|^2}\\
 &\qquad + \sum_{i=1}^N - g_1 \alpha_1 (\xi + \gamma/m_i) \sqrt{\xi^2 +1} \frac{p_i \cdot \nabla_{q_i} U}{|\nabla U|^2}\\
 &\qquad + \sum_{i=1}^N - \frac{2\alpha_1}{p_*} f_1 f_2' f_3 (\xi + \gamma/m_i) |p_i|^2 \frac{p\cdot \nabla U}{|\nabla U|^2}.  
 \end{align*}
  Recalling that $c= \min_i m_i>0$, note that we can estimate each term above as follows
  \begin{align*}
  \mathscr{A} \psi_1&\leq - g_1 \alpha_1 \sqrt{ \xi^2 + 1} + \frac{4\gamma \alpha_1 + 2\gamma \alpha_1 k N}{\beta p_*} f_1 |f_2'| f_3 \frac{|p|}{|\nabla U|}\\
  &\qquad + \frac{8 \gamma \alpha_1}{\beta p_*} f_1 |f_2''| f_3 \frac{|p|}{|\nabla U|} + \frac{2\alpha_1}{p_*} f_1 |f_2'| f_3 |p|^2\\
  &\qquad + g_1  \alpha_1 \frac{(|\xi| + \gamma/c) \sqrt{\xi^2 +1}}{|\nabla U|} |p| + \frac{2\alpha_1 }{p_*} f_1 |f_2'| f_3 \frac{( |\xi |+ \gamma/c)|p|}{|\nabla U|} |p|^2.  
  \end{align*}
  Now for any $\epsilon, \alpha_1,  K_* >0$, by first picking $p_*>0$ large enough and then picking $U_*>0$ large enough, we arrive at the global estimate
  \begin{align}
  \label{psi1:b2}
  \mathscr{A} \psi_1 \leq - g_1 \alpha_1 \sqrt{ \xi^2 + 1} + \epsilon  f_0 |\xi| \|p\|^2_m + \epsilon f_0 |\xi| + \epsilon |p|^2 + \epsilon.
  \end{align}
  Finally, we estimate $\mathscr{T}_2 \psi_1$.  Note that
  \begin{align*}
  \mathscr{T}_2 \psi_1(q,p, \xi) &= \frac{\alpha_1 p \cdot \nabla U}{a |\nabla U|^2} (\|p\|_m^2 - kN/\beta) \bigg\{f_1' f_2 f_3 \sqrt{\xi^2 +1} - f_1 f_2' f_3 \frac{|p|^2 \xi}{p_* (\xi^2 + 1)^{3/2}} \\
  &\qquad \qquad   - f_1 f_2 f_3' \frac{2|\nabla U|^2 \xi}{U_*(\xi^2 +1)^{2}} + g_1 \frac{\xi}{\sqrt{\xi^2 +1}} \bigg\}\\
  &\leq (\|p\|_m^2 + k N /\beta) \bigg\{ \frac{\alpha_1}{a} |f_1'| f_2 f_3 \frac{|p| \sqrt{\xi^2 + 1}}{|\nabla U|} + \frac{2\alpha_1}{a} f_1 |f_2'| f_3 \frac{|p| |\xi|}{|\nabla U| (\xi^2 +1)}\\
  &\qquad \qquad \frac{4 \alpha_1}{a} f_1 f_2 |f_3'| \frac{|p| |\xi|}{|\nabla U| (\xi^2 + 1)} + \frac{\alpha_1}{a} \frac{|p| |\xi|}{|\nabla U |\sqrt{\xi^2 + 1}|}\bigg\}
  \end{align*}
Thus by counting powers of $|\xi|$, for any choice of $\epsilon, \alpha_1, p_*, K_*>0$, we may pick $U_*>0$ large enough so that  
 \begin{align}
 \label{psi1:b3}
 \mathscr{T}_2 \psi_1 (q, p, \xi) \leq \epsilon f_0 | \xi| \|p\|_m^2 + \epsilon f_0 |\xi|+ \epsilon |p|^2 + \epsilon.  
 \end{align} 
 Combining the bounds~\eqref{psi1:b1}, \eqref{psi1:b2} and \eqref{psi1:b3} and adjusting $\epsilon>0$ and the constants appropriately produces the estimate~\eqref{eqn:psi1est1}.
 
 To establish the bound on $\nabla_p (\psi_1)$ observe that 
 \begin{align*}
 \nabla_p\psi_1 = g_1 \alpha_1 \sqrt{\xi^2 +1} \frac{\nabla U}{|\nabla U|^2} + f_1 f_2'f_3 \frac{2\alpha_1 p}{p_*} \frac{p \cdot \nabla U}{|\nabla U|^2},
 \end{align*}   
 hence 
 \begin{align*}
 |\nabla_p \psi_1| \leq  g_1\alpha_1  \frac{\sqrt{\xi^2 +1}}{|\nabla U|} + f_1 |f_2'| f_3 \frac{2\alpha_1 |p|^2}{p_* |\nabla U|}.
 \end{align*}
 Thus for every $\epsilon, \alpha_1, p_*, K_*>0$ we may pick $U_*>0$ large enough so that the estimate~\eqref{eqn:psi1est2} holds.  
 \end{proof}
 %%%%%%%%%%%%%%%%%%%%%%%%%%%%
 
 %%%%%%%%%%%%%%%%%%%%%%%%%%%%
 %%%%%%%%%%%%%%%%%%%%%%%%%%%%%

 \subsection{Perturbation $\psi_2$}
In order to define the second perturbation $\psi_2$, let $\alpha_2>0$ be a parameter and $F:\R\rightarrow \R$ be given by 
\begin{align}
F(z)= -2\alpha_2 \int_0^z \exp(-y^2) \int_0^y \exp(x^2) \, dx \, dy.    
\end{align}
Observe that $F'/(-2\alpha_2)=D$ where $D$ is Dawson's integral, as defined and discussed in Remark~\ref{rem:dawson}.  The function $F$ will make up part of the formula for $\psi_2$, and thus we will need the following proposition in our analysis below.  
\begin{proposition}
\label{prop:asymp}
As $|z|\rightarrow \infty$
\begin{align}
\frac{|F(z)|}{\alpha_2 \log(|z|)} \rightarrow 1\qquad \text{ and } \qquad \frac{|z| |F'(z)|}{\alpha_2}\rightarrow 1.   
\end{align}
\end{proposition}
\begin{proof}
By symmetry, it suffices to prove the asymptotic formulas as $z\rightarrow \infty$.  Note that for $z\geq 1$ we may write 
\begin{align*}
F(z)&= - 2\alpha_2 \int_1^z \exp(-y^2) \int_1^y \exp(x^2) \, dx \, dy +R_0(z)\\
&=-\alpha_2 \log(z) - \alpha_2 \int_1^z \exp(-y^2) \int_1^y \frac{\exp(x^2)}{x^2}\, dx \, dy + R_1(z)\\
&= -\alpha_2 \log(z) +R_2(z)
\end{align*}   
where the penultimate line follows by integration by parts on 
\begin{align*}
\int_1^y \exp(x^2) \, dx= \int_1^y \frac{1}{2x}(\exp(x^2))' \, dx.
\end{align*}
It is not hard to check that $R_2(z)=o(\log(z))$ and $R_2'(z)= o(1/z)$ as $z\rightarrow \infty$, finishing the proof.       
\end{proof}

Fix $i \in \{1,2,\ldots, N\}$ and a constant $\xi_*>\max_{j} (3\gamma/m_j+1)$.  Let $h_2=f_2$ and introduce auxiliary cutoff functions $ h_i \in C^{\infty}( \R; [0,1])$, $i=1,3$, satisfying 
\begin{align*}
h_1(y) = \begin{cases}
1 & \text{ if } y \leq - \xi_* -1 \\
0 & \text{ if } y \geq -\xi_*
\end{cases} \qquad \text{ with } \qquad |h_1'| \leq 2
\end{align*}
and
\begin{align*}
h_3(y) = \begin{cases}
1 & \text{ if } \,\, |y| \leq 3 \\
0 & \text{ if } \,\, |y| \geq 4
\end{cases}
\end{align*}
Recalling the parameters $p_*, U_*>0$, define
\begin{align*}
g_2(q,p, \xi) = h_1(\xi) h_2\bigg( \frac{|p|^2}{p_*\sqrt{ \xi^2 +1}}\bigg) h_3 \bigg( \frac{|\nabla U (q)|^2}{U_*( \xi^2 + 1)}\bigg).
\end{align*}  
Fixing $\ell\in \{1,2,\ldots, k\}$ we set 
\begin{align*}
\psi_i^\ell(q,p, \xi) = \begin{cases}
g_2(q,p, \xi) F\bigg(\frac{|\xi+ \tfrac{\gamma}{m_i}|^{1/2}}{\sqrt{2\gamma/\beta} }\bigg( p_i^\ell- \frac{\partial_{q_i^\ell} U}{|\xi+\frac{\gamma}{m_i}|} \bigg)\bigg) & \text{ if } \xi \leq -\frac{3\gamma}{m_i}\\
0 & \text{ if } \xi > -\frac{3\gamma}{m_i}
\end{cases}
\end{align*}
and define $\psi_2: \mathcal{X}\rightarrow \R$ by 
\begin{align}
\psi_2(q,p, \xi) = \sum_{i=1}^N \sum_{\ell=1}^k \psi_i^\ell(q,p, \xi).  
\end{align}
We now show the following:
\begin{lemma}
\label{lem:psi2}
For each $\epsilon, \alpha_2, K_* >0$ we can pick $p_*>0$ large enough, then $U_* >0$ large enough and then $\xi_* >0$ large such that the following estimates hold on $\mathcal{X}$
\begin{align}
\label{eqn:psi2est0}
|\psi_2| \leq \epsilon H,
\end{align}
\begin{align}
\label{eqn:psi2est}
\mathscr{L} \psi_2(q,p, \xi) \leq - \alpha_2 kN  g_2  |\xi| + \epsilon |p|^2 + \epsilon + \epsilon f_0 |\xi | \|p\|_m^2 + \epsilon f_0 |\xi| ,
\end{align} 
and
\begin{align}
\label{eqn:gradpsi2}
\frac{\gamma}{\beta}|\nabla_p (\psi_2)|^2& \leq  2 k N  \alpha_2^2 D_\text{\emph{max}}^2 g_2     |\xi| + \epsilon f_0 |\xi| + \epsilon.  
\end{align}
\end{lemma}
\begin{proof}
To see the first estimate~\eqref{eqn:psi2est0}, fix $i \in \{ 1,2,\ldots, N\}$ and $\ell\in \{ 1,2,\ldots, k \}$.  Applying Proposition~\ref{prop:asymp}, we see that there exists constants $C, D>0$ such that for all $(q,p, \xi) \in \mathcal{X}$
\begin{align*}
|\psi_i^\ell(q,p, \xi)| \leq g_2(C \log|\xi|+D).
\end{align*}
Since $0\leq g_2 \leq 1$ globally and $g_2\equiv 0$ whenever $\xi \geq -\xi_*$, it follows that we may pick $\xi_*>0$ large enough so that~\eqref{eqn:psi2est0} holds.

To estimate $\mathscr{L}\psi_2$, we again fix $i \in \{ 1,2,\ldots, N\}$ and $\ell\in \{ 1,2,\ldots, k \}$ and estimate $\mathscr{L} \psi_i^\ell$, showing that for each $\epsilon, \alpha_2, K_*>0$, we may pick $p_*>0$ large enough, then $U_*>0$ large enough and then $\xi_*>0$ large enough so that the estimate
\begin{align}
\label{eqn:psilmest}
\mathscr{L} \psi_i^\ell (q,p, \xi) \leq - \alpha_2 g_2 |\xi| + \epsilon |p|^2 + \epsilon + \epsilon f_0 |\xi| \|p\|_m^2 + \epsilon f_0 |\xi|
\end{align}
holds on $\mathcal{X}$.  As in the previous lemma, we again break up $\mathscr{L} \psi_i^\ell$ as
\begin{align*}
\mathscr{L} \psi_i^\ell= \mathscr{T}_1 \psi_i^\ell + \mathscr{A} \psi_i^\ell + \mathscr{T}_2 \psi_i^\ell
\end{align*}
where $\mathscr{T}_1, \mathscr{T}_2$ and $\mathscr{A}$ were introduced either just above or in equation~\eqref{eqn:opdef}.  Even though $F$ in the definition of $\psi_i^\ell$ depends on $i,\ell$, we will suppress this dependence for simplicity.  Beginning with $\mathscr{T}_1 \psi_i^\ell$ observe that 
\begin{align*}
\mathscr{T}_1 \psi_i^\ell &= \sum_{n, j} \bigg\{- g_2 F'  \frac{p_n^j \partial^2_{q_n^j q_i^\ell} U}{\sqrt{2\gamma/\beta |\xi+ \gamma/m_i |}} + \sum_{s,t} h_1 h_2 h_3' F \frac{2 p_n^j \partial_{q_s^t} U \partial^2_{q_n^j q_s^t }U}{U_* (\xi^2 +1)}\bigg\}\\
&\leq g_2  \frac{|F'||p| |\nabla^2 U|}{\sqrt{2\gamma/\beta  |\xi + \gamma/m_i |}} + 2 h_1 h_2 |h_3'| \frac{|F| |\nabla U| |\nabla^2 U| |p| }{U_* (\xi^2 +1)}\\
&\leq g_2  \frac{|F'||p| |\nabla^2 U|}{\sqrt{2\gamma/\beta  |\xi + \gamma/m_i |}}1_{\{U\leq R\}} +g_2  \frac{|F'||p| |\nabla U|^\zeta}{\sqrt{2\gamma/\beta  |\xi + \gamma/m_i |}} \frac{|\nabla^2 U|}{|\nabla U|^\zeta}1_{\{ U \geq R\}} \\
&\qquad + 2 h_1 h_2 |h_3'| \frac{|F| |\nabla U| |\nabla^2 U| |p| }{U_* (\xi^2 +1)} 1_{\{U \leq R\}} + 4 h_1 h_2 |h_3'| |F|  |\nabla U|^{\zeta-1} |p|   \frac{|\nabla^2 U|}{|\nabla U|^\zeta}1_{\{U\geq R\}}  
\end{align*}
where $\zeta\in (1,2)$ is the constant in (A4) of Definition~\ref{def:normpot}.  Applying the asymptotics in (A4) and those for $F$ in Proposition~\ref{prop:asymp}, for each $\epsilon >0$ we may pick $R>0$ large enough so that 
\begin{align*}
\mathscr{T}_1 \psi_i^\ell &\leq  C_1 g_2  \frac{|p|}{|\xi+ \gamma/m_\ell|^{1/2}} + \epsilon g_2 |p| |\xi|^{1/2} + 2 C_2 h_1 h_2 |h_3'| \frac{|p|}{|\xi|} + \epsilon h_1 h_2 |h_3'| |\xi| |p| 
\end{align*}
for some constants $C_i>0$ depending on $\alpha_2, p_*, U_*, \epsilon, R$.  Thus for all $\epsilon, \alpha_2, p_*, U_* >0$, picking $\xi_*>0$ large enough we can arrive at the inequality
\begin{align}
\mathscr{T}_1 \psi_i^\ell \leq \epsilon + \epsilon f_0 |\xi|  + \epsilon f_0  |\xi| \|p\|_m^2 .  
\end{align}

 Turning now to  $\mathscr{A} \psi_i^\ell$, for each $n, j$ let
\begin{align*}
\mathscr{A}_n^i = (|\xi+ \gamma/m_n| p_n^j- \partial_{q_n^j} U)\partial_{p_n^j} + \frac{\gamma}{\beta} \partial_{p_n^j}^2
\end{align*}
and note 
\begin{align*}
\mathscr{A} \psi_i^\ell&= \mathscr{A}_i^\ell \psi_i^\ell + \sum_{(n,j) \neq (i, \ell)} F \mathscr{A}_n^j g_2\\
&= g_2 \mathscr{A}_{i}^\ell (F) + \frac{2\gamma}{\beta}  \partial_{p_i^\ell}(g_2) \partial_{p_i^\ell} (F)+ \sum_{n,j} F \mathscr{A}_n^j g_2  \\
&= - \alpha_2 g_2 |\xi+ \gamma/m_i| +\sqrt{2\gamma/\beta} h_1 h_2' h_3 F' \frac{2 p_i^\ell |\xi+ \gamma/m_i|^{1/2}}{p_* \sqrt{\xi^2 + 1}} + \sum_{n,j} F \mathscr{A}_n^j g_2 .
\end{align*}
Now,
\begin{align*}
F\mathscr{A}_n^j g_2 &= F h_1 h_2' h_3 \frac{2p_n^j}{p_*\sqrt{\xi^2 + 1}} ( |\xi+\gamma/m_n|p_n^j - \partial_{q_n^j} U) +  \frac{\gamma}{\beta}Fh_1 h_2'' h_3 \frac{4(p_n^j)^2}{p_*^2(\xi^2 + 1)}  \\
&\qquad + \frac{\gamma}{\beta} F h_1 h_2' h_3 \frac{2}{p_*\sqrt{ \xi^2 + 1}} \\
& \leq C_1 h_1 |h_2'| h_3| |p| |\xi+ \gamma/m_n|^{1/2}\log|\xi+ \gamma|   + C_2(h_1 |h_2'| h_3 + h_1 |h_2''| h_3)\frac{\log|\xi+ \gamma/m_n|}{\sqrt{\xi^2 + 1}} 
\end{align*}
for some constants $C_i>0$ depending on $\alpha_2, p_* , U_*$.  Thus for every $\epsilon, \alpha_2, p_*, U_*, K_* >0$, picking $\xi_* >0 $ large enough produces the global estimate
\begin{align}
\mathscr{A} \psi_i^\ell &\leq - \alpha_2 g_2 |\xi|   +\epsilon  f_0 \|p\|_m^2 |\xi|  + \epsilon f_0 |\xi|  + \epsilon |p|^2 + \epsilon.   
\end{align}

Lastly, we consider $\mathscr{T}_2 \psi_i^\ell$ and note
\begin{align*}
\mathscr{T}_2 \psi_i^\ell &= a^{-1}(\|p\|^2_m - k N k_B T) \times \bigg\{ h_1' h_2 h_3 F - h_1 h_2' h_3 F \frac{ \xi |p|^2}{p_*(\xi^2 +1)^{3/2}}\\
& - 2h_1 h_2 h_3' F \frac{\xi |\nabla U|^2 }{U_* (\xi^2+ 1)^{2}} - \frac{g_2}{2 |\xi + \gamma/m_i |} F' \frac{|\xi+\gamma/m_i |^{1/2}}{\sqrt{2\gamma/\beta}} \big(p_i^m -\frac{\partial_{q_i^\ell} U}{|\xi+ \gamma/m_i|}\big)  \\
&\qquad \qquad \qquad - \frac{g_2 F'\partial_{q_i^\ell} U}{\sqrt{2\gamma/\beta} |\xi+\gamma/m_i| ^{3/2}} \bigg \}.
\end{align*}
Hence, 
\begin{align*}
\mathscr{T}_2 \psi_\ell^m &\leq a^{-1}(\|p\|^2_m + kN k_B T) \bigg\{ |h_1'| h_2 h_3 |F|+ 2 h_1 |h_2'| h_3 \frac{|\xi|}{\xi^2 + 1} + 6 h_1 h_2 |h_3'| \frac{|\xi|}{\xi^2 + 1}\\
&\qquad \qquad \qquad \qquad \qquad  \qquad + C_1 \frac{g_2}{|\xi + \gamma/m_\ell|} + C_2 g_2 \bigg\}
\end{align*}
for some constant $C_1>0$ depending on $\alpha_2$ and some constant $C_2>0$ depending on $\alpha_2, U_*$. 
Picking $\xi_* >0$ large enough as before, we can thus arrive at the estimate
\begin{align}
\mathscr{T}_2 \psi_i^\ell \leq \epsilon f_0 |\xi| \|p\|^2_m + \epsilon f_0 |\xi|.  
\end{align}
Putting the estimates together we arrive at the claimed inequality~\eqref{eqn:psilmest}.  

Finally, we turn to estimating $\frac{\gamma}{\beta}|\nabla_p(\psi_2)|^2$.  Note that for each $i, j \in \{1,2,\ldots, N\}$ and each $n, \ell \in \{ 1,2,\ldots, k\}$ we have that 
\begin{align*}
\partial_{p_j^{n}} ( \psi_i^\ell) = \delta_{(i, \ell)}(j,n) g_2 F' \frac{|\xi + \gamma/m_i |^{1/2}}{\sqrt{2\gamma/\beta}} + h_1 h_2' h_3 F \frac{2 p_j^n }{p_* \sqrt{\xi^2 + 1}}
\end{align*}
where $\delta_{(i,\ell)}(j,n)=1$ if $(j,n)=(i, \ell)$ and $0$ otherwise.  Hence
\begin{align*}
|\partial_{p_j^n} (\psi_2)| \leq g_2 \frac{2 \alpha_2 D_\text{max}}{\sqrt{2\gamma/\beta}}   |\xi+ \gamma/m_j|^{1/2} + h_1 |h_2'| h_3 \frac{4 (C\log|\xi|+ D)}{p_*^{1/2} (\xi^2 + 1)^{1/4}}
\end{align*}
for some constants $C,D>0$.  
Thus applying Proposition~\ref{prop:asymp} for every $\epsilon >0$ by picking $\xi_*>0$ large as before we find that 
\begin{align*}
|\partial_{p_j^n} (\psi_2)| \leq  g_2 \frac{2 \alpha_2 D_\text{max}}{\sqrt{2\gamma/\beta}}   |\xi+ \gamma/m_j|^{1/2} + \epsilon 
\end{align*}
for all $j,n$.  Hence applying this bound and by possibly increasing $\xi_*>0$ if necessary we arrive at the claimed estimate~\eqref{eqn:gradpsi2}.    

\end{proof}

We now combine the previous estimates to prove Theorem~\ref{thm:maintech}.  Whenever we need to adjust the parameters as done in the statement of Lemma~\ref{lem:psi2}, in the proof below we will simply say ``by adjusting the parameters".   
\begin{proof}[Proof of Theorem~\ref{thm:maintech}]
Let $\beta_0>0$ satisfy~\eqref{eqn:beta0bound} and fix $\epsilon_0 \in (0, \beta_0)$.  Pick $\delta >0$ such that 
\begin{align*}
\delta = \min \bigg\{ \frac{\epsilon_0}{3}, \frac{\beta_*}{2} - \frac{\beta_0}{2}\bigg\}  
\end{align*}
 where we recall the constant $\delta>0$ was introduced above~\eqref{eqn:ham}.  Note that by adjusting the parameters, the estimates~\eqref{eqn:psi0bound0},~\eqref{eqn:psi1est0} and~\eqref{eqn:psi2est0} together imply part (i) of the result.  

To establish condition (ii) of the result, first note that 
\begin{align}
\mathscr{L}(\beta_0 H) = - \gamma \beta_0 \sum_{i=1}^N \frac{|p_i|^2}{m_i^2}- \frac{\beta_0}{\beta}  kN  \xi + \frac{\beta_0}{\beta}  K_1,
\end{align}
where $K_1 :=\gamma \sum_{i} m_i^{-1}>0$, and
\begin{align}
 \nabla_{p_i} (\beta_0 H) =\beta_0  \frac{p_i}{m_i}.  
\end{align}
Let $V=\beta_0 H  + \psi_0 + \psi_1 + \psi_2$.  Combining relations~\eqref{eqn:psi0bound1} and \eqref{eqn:psi0bound2} with Lemma~\ref{lem:psi1} and Lemma~\ref{lem:psi2}, for each $\epsilon>0$ we may vary the parameters so that 
\begin{align}
\label{eqn:LVineq}
\mathscr{L}V(q,p, \xi)& \leq  - \gamma \beta_0(1- \epsilon) \sum_{i=1}^N \frac{|p_i|^2}{m_i^2} - \frac{\beta_0}{\beta} kN \xi  - f_0 \delta(1-\epsilon) |\xi| \|p\|_m^2  \\
\nonumber &\qquad - g_1 \alpha_1 \sqrt{\xi^2 +1} - \alpha_2 k N g_2 |\xi| + f_0 \frac{\delta + \epsilon}{\beta} kN |\xi| + \frac{\beta_0}{\beta} K_1 + \frac{\delta+ \epsilon}{\beta} k N    
\end{align}
as well as the inequality 
\begin{align*}
\frac{\gamma}{\beta}| \nabla_p V|^2 &\leq\frac{\gamma}{\beta}\bigg\{ |\nabla_p(\beta_0 H)|^2 + |\nabla_p(\psi_1)|^2 + |\nabla_p(\psi_2)|^2 + 2 |\nabla_p(\beta_0 H) | |\nabla_p(\psi_1)| \\
&\qquad + 2 |\nabla_p(\beta_0 H)| |\nabla_p(\psi_2)| + 2|\nabla_p(\psi_1) | |\nabla_p(\psi_2)|\bigg\}\\
& \leq \frac{\gamma \beta_0^2 }{\beta}\sum_{i=1}^N \frac{|p_i|^2}{m_i^2}  + 2  kN   \alpha_2^2 D_\text{max}^2 g_2     |\xi| + \epsilon  f_0 |\xi| + \epsilon f_0 \| p \|^2_m |\xi|  + \epsilon  + \epsilon |p|^2.   
\end{align*}
Combining the previous inequality with~\eqref{eqn:LVineq} and adjusting $\epsilon>0$ and the parameters accordingly produces the estimate
\begin{align}
\label{eqn:LVGV1}&\mathscr{L}V(q,p, \xi) + \frac{\gamma}{\beta}|\nabla_p(V)|^2 \\
\nonumber &\qquad \leq -\gamma \beta_0 (1- \beta_0/\beta-\epsilon) \sum_{i=1}^N \frac{|p_i|^2}{m_i^2} - \frac{\beta_0}{\beta} kN  \xi - f_0 \delta (1-\epsilon) |\xi| \|p\|_m^2\\
\nonumber &\qquad - g_1 \alpha_1 \sqrt{\xi^2 + 1} - g_2 \alpha_2  kN (1-2 \alpha_2 D_\text{max}^2) |\xi|  + f_0 \frac{\delta + \epsilon}{\beta} kN |\xi|  + \frac{\beta_0}{\beta} K_1+\frac{\delta+ \epsilon}{\beta} kN.    
\end{align}    
Let $\alpha>0$ and pick the rest of the parameters as follows:
\begin{align}
\label{eqn:paramc}
&\epsilon < \min\bigg\{ \frac{1}{2}- \frac{\beta_0}{2\beta}, \frac{1}{2}, \frac{\beta}{8D_\text{max}^2} - \beta_0 -\delta \bigg\}, \,\, \alpha_1 = 2 \alpha + 2\frac{\beta_0}{\beta} K_1 + 2\frac{\beta_0 +\delta + \epsilon}{\beta} kN, \\\ &\alpha_2 = \frac{1}{4 D_\text{max}^2} \qquad \qquad  K_* = \frac{\beta \alpha}{\beta_0 kN}+  \frac{K_1}{kN} + \frac{ \beta_0 + \delta + \epsilon}{\beta_0}.
\end{align}
Applying these choices to the estimate~\eqref{eqn:LVGV1} then gives  
\begin{align*}
\mathscr{L}V(q,p, \xi) + \frac{\gamma}{\beta}|\nabla_p(V)|^2 &\leq - \frac{\gamma \beta_0}{2}(1- \beta_0/\beta) \sum_{i=1}^N \frac{|p_i|^2}{m_i^2} - \frac{\beta_0}{\beta} k N |\xi|(1-f_0)- \frac{f_0}{2}\delta |\xi| \|p\|_m^2\\
&\qquad - g_1 \frac{\alpha_1}{2}|\xi| - g_1 \frac{\alpha_1}{2} - g_2 \frac{kN}{8D_\text{max}^2} |\xi|  + f_0 \frac{\beta_0+ \delta + \epsilon}{\beta} kN  |\xi| \\
&\qquad  + \frac{\beta_0}{\beta} K_1 + \frac{ \beta_0 + \delta + \epsilon}{\beta} kN
\end{align*} 
where we have left $\alpha_1$ as is for brevity of mathematical expression.  First observe that if $\xi \geq K_*$, then 
\begin{align*}
\mathscr{L}V(q,p, \xi) + \frac{\gamma}{\beta}|\nabla_p(V)|^2 &\leq  - \frac{\beta_0}{\beta} kN K_* +  \frac{\beta_0}{\beta} K_1 + \frac{\beta_0 + \delta + \epsilon}{\beta} kN \leq -\alpha.    
\end{align*} 
Also note that 
\begin{align*}
\mathscr{L}V(q,p, \xi) + \frac{\gamma}{\beta}|\nabla_p(V)|^2 &\leq  - \frac{\gamma \beta_0}{2}(1- \beta_0/\beta) \sum_{i=1}^N \frac{|p_i|^2}{m_i^2}- \frac{f_0}{2}\delta |\xi| \|p\|_m^2 \\
&\qquad + f_0 \frac{\beta_0+ \delta + \epsilon}{\beta} kN  |\xi| + \frac{\beta_0}{\beta} K_1 + \frac{\beta_0 + \delta + \epsilon}{\beta} kN.  
\end{align*}
Hence if $c= \min_i m_i$ and $|p|\geq P>0$ for any $P>0$ sufficiently large we have 
then we have 
\begin{align*}
\mathscr{L}V(q,p, \xi) + \frac{\gamma}{\beta}|\nabla_p(V)|^2 &\leq -\alpha.  
\end{align*}
Thus now suppose that $|p| \leq P$ and $-\xi_*-3\leq \xi \leq K_*$.  This means that both $|p|$ and $\xi$ are both  bounded.  Thus the only possibility for $H\rightarrow \infty$ is if $U(q)\rightarrow \infty$.  Thus, for $U(q)$ large enough in this region, $g_1=1$, $f_0=1$ and  
\begin{align*}
\mathscr{L}V(q,p, \xi) + \frac{\gamma}{\beta}|\nabla_p(V)|^2 &\leq - \frac{\alpha_1}{2}|\xi| - \frac{\alpha_1}{2} +
 \frac{\beta_0+ \delta + \epsilon}{\beta} kN  |\xi| \\
 &\qquad + \frac{\beta_0}{\beta} K_1 + \frac{\beta_0 + \delta + \epsilon}{\beta} kN \\
&\leq -\alpha
\end{align*}  
where in the last inequality we used the choice of $\alpha_1$ in~\eqref{eqn:paramc}.  Finally, if $|p|\leq P$ and $\xi \leq -\xi_*-3$, then $g_2=1$, $f_0=1$ and 
\begin{align*}
\mathscr{L}V(q,p, \xi) + \frac{\gamma}{\beta}|\nabla_p(V)|^2 &\leq - \frac{kN }{8D_\text{max}^2} |\xi| +  \frac{\beta_0+ \delta + \epsilon}{\beta} kN  |\xi| \\
&\qquad  + \frac{\beta_0}{\beta} K_1 + \frac{ \beta_0 + \delta + \epsilon}{\beta} kN
\end{align*} 
By choice of $\delta, \epsilon>0$,  we observe that $\frac{\beta_0+ \delta + \epsilon}{\beta}< \frac{1}{8D_\text{max}^2}$, so by increasing $\xi_*>0$ is necessary, we also arrive at the estimate
\begin{align*}
\mathscr{L}V(q,p, \xi) + \frac{\gamma}{\beta}|\nabla_p(V)|^2 &\leq -\alpha
\end{align*}
in the region $|p| \leq P$ and $\xi\leq -\xi_*-3$.  This finishes the proof.  
\end{proof}

\section{Smoothing and Support Properties} 
\label{sec:sup_con}

Here we establish conclusions (i) and (ii) of Proposition~\ref{prop:suppreg} separately.

\begin{proof}[Proof of Proposition~\ref{prop:suppreg} (i)]
We apply Corollary~7.2 of~\cite{RB_06} and check that H\"{o}rmander's bracket condition, as stated in relation (162) of~\cite{RB_06}, is satisfied on $\mathcal{X}$.  See~\cite{Hor_67} for H\"{o}rmander's original statement and proof.  Letting $\mathscr{L}^*$ denote the formal $L^2$-adjoint of $\mathscr{L}$, this will then ensure hypoellipticity of the operators $\mathscr{L}, \mathscr{L}^*, \partial_t \pm \mathscr{L}, \partial_t \pm \mathscr{L}^*$ on the respective domains $\mathcal{X}, \mathcal{X}, (0, \infty) \times \mathcal{X}, (0, \infty) \times \mathcal{X}$.  For $i\in \{ 1,2, \ldots, N\}$ and $\ell\in \{ 1, 2, \ldots, k\}$, let $X_i^\ell = \partial_{p_i^\ell} $ and define
\begin{align*}
X_0= \sum_{i=1}^N \frac{p_i}{m_i}\cdot \nabla_{q_i} - \sum_{i=1}^N \big( \xi + \frac{\gamma}{m_i}\big) p_i \cdot \nabla_{p_i}   - \nabla U \cdot \nabla_p + a^{-1}(\|p\|_m^2 - kN/\beta) \partial_\xi .
\end{align*}
Letting $[A,B]=AB-BA$ denote the commutator of operators $A$ and $B$, we find that 
\begin{align*}
[X_i^\ell, X_0]&=m_{i}^{-1} \partial_{q_i^\ell} - (\xi + \gamma/m_i) \partial_{p_i^\ell} + \frac{2}{a m_i} p_{i}^\ell \partial_\xi 
\end{align*}
and $[X_i^\ell, [X_i^\ell, X_0]] = \frac{2}{m_i} \partial_\xi$.  Therefore, the list of vector fields
\begin{align*}
&X_i^\ell, \qquad \qquad \qquad  i=1,2, \ldots, N, \,\, \ell=1,2,\ldots, k\\
&[X_i^\ell, X_0 ],\qquad  \qquad i=1,2, \ldots, N, \,\,  \ell=1,2, \ldots, k\\
&[X_i^\ell, [X_i^\ell, X_0]],\qquad  i=1,2, \ldots, N, \,\, \ell=1,2, \ldots, k
\end{align*}  
has full rank at every point $x\in \mathcal{X}$.    
\end{proof}

We next turn to the proof of Proposition~\ref{prop:suppreg} (ii) which relies on the support theorems~\cite{SV_72, SV_73}.  That is, to equation~\eqref{eqn:main} we associate a deterministic control problem on $\mathcal{X}$
\begin{align}
\label{eqn:controlsystem}
\dot{Q}_i&= P_i\\
\nonumber \dot{P}_i&=  -\Xi P_i - \frac{\gamma}{m_i} P_i - \nabla_{Q_i} U(Q) + \sqrt{2\gamma/\beta}\,  \eta_i \\
\nonumber  \dot{\Xi}&= \sum_{i=1}^N \frac{|P_i|^2}{a m_i}  -  \frac{kN}{a\beta}   
\end{align}       
where $\eta=(\eta_i)$ is a piecewise continuous $(\R^k)^N$-valued control and $\beta=1/(k_B T)$.  Intuitively for a fixed such $\eta$, the solution of~\eqref{eqn:controlsystem} represents an approximate sample trajectory of the solution of~\eqref{eqn:main}.  The support theorems~\cite{SV_72, SV_73} make this intuition precise.  In particular, for $x\in \mathcal{X}$ and $t>0$ define $A(x,t)$ to be the set of points $y\in \mathcal{X}$ such that there exists a piecewise continuous $(\R^k)^N$-valued control $\eta= (\eta_i)$ for which the solution of~\eqref{eqn:controlsystem} exists on the time interval $[0,t]$ in $\mathcal{X}$ and has $(Q(0), P(0), \Xi(0))=x$ and $(Q(t), P(t), \Xi(t)) =y$.  Then the support theorems~\cite{SV_72, SV_73} imply that for every $x\in \mathcal{X}$ and every $t>0$
\begin{align}
\label{eqn:suppeq}
\supp \mathscr{P}_t(x, \, \cdot \,) = \text{closure}(A(x,t)).  
\end{align}      
Thus the problem of finding points in $\supp \mathscr{P}_t(x, \, \cdot \,)$ can be cast in terms finding points reachable from the system~\eqref{eqn:controlsystem} at exactly time $t>0$ started at $x\in \mathcal{X}$ as the controls vary through the class of piecewise continuous functions. 

\begin{remark}
It is worth noting that solving the control problem above is slightly more involved than the one in the case of Langevin dynamics with uniformly elliptic noise in the momentum directions.  One can see the difference between the two cases almost immediately, as the process solving~\eqref{eqn:main} is not fully supported in $\mathcal{X}$ instantaneously.  Indeed, for all $t\geq 0$
\begin{align*}
\xi(t) \geq \xi(0) - ta^{-1} kN/\beta \qquad \P-\text{a.s.}  
\end{align*}      
In other words, the $\xi$ process is only allowed to decrease so fast, hence restricting access to points at a given time $t>0$ sufficiently far to the left of where it started.  Although from this observation one is tempted to conjecture that for a given $x=(q,p,\xi) \in \mathcal{X}$ and $t>0$
\begin{align*}
\supp \mathscr{P}_t(x, \, \cdot \,) = \{(q',p', \xi') \in \mathcal{X} \, : \,  \xi' \geq \xi  - ta^{-1} k_B T kN\},
\end{align*}
this is false.  Consult the statement of Proposition~\ref{prop:suppreg} (ii) to see what precisely is claimed.  Thus determining the supports of the transitions, and hence solving the control problem~\eqref{eqn:controlsystem}, is more subtle.     
\end{remark}

While there are other methods, such as those from geometric control theory, that could prove useful in analyzing the problem above (see, for example, \cite{GHM_17, Jur_97} and the Agrachev-Sarachev approach as outlined in the infinite-dimensional setting in~\cite{Shi_14}), we choose to prove Proposition~\ref{prop:suppreg} (ii) by an essentially explicit construction.  As seen below, we can re-cast the control problem as a calculus of variations problem.        
  
\begin{proof}[Proof of Proposition~\ref{prop:suppreg} \emph{(ii)}]
Let $x=(q,p, \xi)\in \mathcal{X}$ and $t>0$.  By the support theorems~\cite{SV_72, SV_73}, it suffices to prove that 
\begin{align}
\label{eqn:suppP}
\text{closure}(A(x,t)) = \mathcal{A}(x,t)
\end{align}  
where $\mathcal{A}(x,t)$ is as in~\eqref{eqn:suppset}.  To see the inclusion $``\subseteq"$ in~\eqref{eqn:suppP}, let $\eta=(\eta_i)$ be an arbitrary piecewise continuous $(\R^k)^N$-valued control such that the solution of~\eqref{eqn:controlsystem} with $(Q(0), P(0), \Xi(0)) = (q,p, \xi)$ exists in $\mathcal{X}$ for all times on $[0,t]$.  Letting $q'= Q(t)$ and $\xi'=\Xi(t)$, note by Jensen's inequality  
\begin{align*}
L^2_{q, q'} \leq \bigg(\int_0^t \| P(u) \|_m  \, du\bigg)^2 \leq t \int_0^t \| P(u) \|_m^2 \, du.  
\end{align*} 
Consequently, 
\begin{align*}
\xi' &= \xi+  a^{-1} \int_0^t \| P(u) \|_m^2 \, du - t a^{-1} k_B T k N\\
& \geq  \xi + \frac{L^2_{q, q'}}{ta}  - t a^{-1} k_B T k N.  
\end{align*}
This establishes the claimed inclusion.  For the other inclusion $``\supseteq"$ in~\eqref{eqn:suppP}, since $\mathcal{O}$ is open let $\epsilon >0$ be small enough so that $B_\epsilon(q) \subseteq \mathcal{O}$ and define $X_\epsilon = B_\epsilon(q) \times (\R^k)^N \times \R$.  It suffices to show that 
\begin{align*}
\text{closure}(A(x,t)) \cap X_\epsilon  \supseteq \mathcal{A}(x,t) \cap X_\epsilon. 
\end{align*}
Note that this allows us to ``convexify" the problem; that is, for any $q' \in B_\epsilon(q)$, $L_{q,q'} = \| q- q'\|_m$.  Let $(q', p', \xi') \in X_\epsilon$ and $t>0$.  For a small parameter $\delta\in (0, t/2)$, consider the following piecewise linear curve $\phi_\delta:[0, t] \rightarrow (\R^k)^N$ given by 
\begin{align*}
\phi_\delta(u) = \begin{cases}
q + u p & \text{ if } 0 \leq u \leq \delta\\
\ell_u(q+\delta p, q' -\delta p') & \text{ if } \delta \leq u \leq t- \delta \\
q' + (u-t) p' & \text{ if } t-\delta \leq u \leq t.  
\end{cases}
\end{align*}
where $u\mapsto \ell_u(q+\delta p, q'-\delta p')$ linearly interpolates between the points $q+ \delta $ at time $u=\delta$ and $q'-\delta p'$ at time $u=t-\delta$.  Observe that for any $\delta >0$ sufficiently small, $\phi_\delta([0,t]) \subseteq B_\epsilon(q)$ and that for every $\delta >0$, $\phi_\delta(0)= q, \dot{\phi}_\delta(0)= p, \phi_\delta(t)=q', \dot{\phi}_\delta(t)= p'$.  Moreover, observe that as $\delta \rightarrow 0$  
\begin{align*}
\int_0^t \| \dot{\phi}_\delta(u) \|_m^2 \, du \rightarrow \frac{\| q - q'\|_m^2}{t} = \frac{L^2_{q,q'}}{t}.  
\end{align*}  
Thus by picking $\eta= (\eta_{i, \delta})$ to satisfy the second equation in~\eqref{eqn:controlsystem} with this choice of $(Q,P)=(\phi_\delta, \dot{\phi}_\delta)$ proves that all points $(q', p', \xi') \in X_\epsilon$ with $$\xi' = \xi + (a t)^{-1}L^2_{q,q'} - t a^{-1}k_B T kN $$ belong to $\text{closure}(A(x,t))$.  To get the remaining points, first suppose that $q\neq q'$.  For $s\in (0,t]$, define
\begin{align*}
\phi_\delta^s(u) = \begin{cases}
q + u p & \text{ if } 0 \leq u \leq \delta\\
\ell_u(q+\delta p, q' -\delta p') & \text{ if } \delta \leq u \leq s- \delta \\
q'- \delta p' & \text{ if } s-\delta \leq u \leq t-\delta\\
q' + (u-t) p' & \text{ if } t-\delta \leq u \leq t  
\end{cases}
\end{align*} 
where in the above, the $u_\mapsto \ell_u(q_1, q_2)$ is the line segment connecting $q_1$ and $q_2$ at times $u=\delta$ and $u=s-\delta$.  
Note that for every $s\in (0, t]$, as $\delta \rightarrow 0$ 
\begin{align*}
\int_0^t \|\dot{\phi}^s_\delta (u) \|_m^2 \, du&\rightarrow \frac{\| q- q'\|_m^2}{s} + \|p'\|_m^2 (t-s):=f(s) .  
\end{align*}   
Since $f$ is a continuous function of $s\in (0,t]$ with $f(t) = \|q-q'\|_m^2/t$ and $\lim_{s\downarrow 0} f(s)=\infty$, we can apply the Intermediate Value Theorem to see that all points $(q', p', \xi') \in X_\epsilon$ with $q'\neq q$ satisfying
\begin{align*}
\xi' > \xi + (a t)^{-1} L_{q,q'}^2 - t a^{-1} k_B T kN
\end{align*} 
belong to $\text{closure}(A(x,t))$.  Since $\text{closure}(A(x,t))$ is closed the result follows.  
\end{proof}

\section*{Acknowledgements}
The author is grateful to Jonathan Mattingly for originally suggesting the idea for the project and acknowledges fruitful conversations about the topic of this paper with Nathan Glatt-Holtz, Scott McKinley and Hung Nguyen.  The author is supported in part by grant DMS-1612898 from the National Science Foundation
\section*{Appendix}

Here we use Theorem~\ref{thm:lyapexist} and Proposition~\ref{prop:suppreg} to conclude Theorem~\ref{thm:main}.  What follows is fairly standard but we provide the details for completeness.  We first translate Theorem~\ref{thm:lyapexist} and Proposition~\ref{prop:suppreg} to the following two corollaries which allow us to better connect with the setup in~\cite{HM_08}.  

\begin{corollary}
\label{cor:lyap}
Let $\beta_0>0$ satisfy~\eqref{eqn:beta0bound}, fix $\epsilon\in (0, \beta_0)$ and let $W\in C^2(\mathcal{X})$ and $R, \alpha, K>0$ satisfy \emph{(I)} and \emph{(II)} in Theorem~\ref{thm:lyapexist}.  Then for all $t\geq 0$ and all $x\in \mathcal{X}$
\begin{align}
\mathscr{P}_t W(x) \leq e^{-\alpha t} W(x) + K/\alpha.  
\end{align}  
\end{corollary}

\begin{proof}
For $n\in \N$, define $\sigma_n= \inf\{ t>0 \, : \, W(x(t)) >n\}$ and let $\sigma_n(t)= t\wedge \sigma_n$.  By construction of $W$, we see that $\E_x e^{\alpha \sigma_n(t)} (W(x_{\sigma_n(t)})-K/\alpha)\leq W(x)- K/\alpha$, which in turn implies the estimate
\begin{align}
\label{eqn:lyapcor1}
\E_x e^{\alpha \sigma_n(t)} W(x_{\sigma_n(t)}) \leq W(x)+ \E_x e^{\alpha \sigma_n(t)} K/\alpha.  
\end{align}       
Note that $\sigma_n\uparrow \infty$ $\P_x$-almost surely since $W(x)\rightarrow \infty$ as $H(x)\rightarrow \infty$ and $x(t)$ is non-explosive.  Applying Fatou's lemma and monotone convergence to~\eqref{eqn:lyapcor1} finishes the proof.     
\end{proof}

\begin{corollary}
\label{cor:minor}
Let $\beta_0>0$ satisfy~\eqref{eqn:beta0bound}, fix $\epsilon\in (0, \beta_0)$ and let $W\in C^2(\mathcal{X})$ and $R, \alpha, K>0$ satisfy \emph{(I)} and \emph{(II)} in Theorem~\ref{thm:lyapexist}.  For $R>0$, define $$\mathcal{C}_R= \{ x\in \mathcal{X} \, : \, W(x) \leq R\}.$$  Then for each $R>0$, $\mathcal{C}_R$ is compact.  Also, for each $R>0$ large enough and each $t_0>0$, there exists a probability measure $\nu$ on Borel subsets of $\mathcal{X}$ and a constant $c>0$ such that for all $A\in \mathcal{B}(\mathcal{X})$ 
\begin{align}
\inf_{x\in \mathcal{C}_R} \mathscr{P}_{t_0}(x, A) \geq c \nu(A)
\end{align}   
\end{corollary}

\begin{proof}
The fact that $\mathcal{C}_R$ is compact for $R>0$ follows since $W(x)\rightarrow \infty$ as $H(x)\rightarrow \infty$, $W\in C^2(\mathcal{X};(0, \infty))$ and since the poetntial $U$ is normal.  Let $R>0$ be large enough so that $\mathcal{C}_R\neq \emptyset$ and fix $t_0>0$.  First observe that for any $A\subseteq \mathcal{X}$ Borel and $x\in \mathcal{X}$ we may write
\begin{align}
\label{eqn:density}
\mathscr{P}_{t_0}(x, A)= \int_A \int_\mathcal{X} r_{\frac{t_0}{2}}(x,y) r_{\frac{t_0}{2}}(y, z) \, dy \, dz
\end{align}   
where we recall that $y\mapsto r_t(x,y)$ denotes the probability density of $\mathscr{P}_t(x, \, \cdot \,)$ with respect to Lebesgue measure on $\mathcal{X}$.  From this expression, the goal is to now use support and regularity properties of the transitions to bound the quantity below by a positive constant times normalized Lebesgue measure on a bounded subset of $\mathcal{X}$.  To this end, since $\mathcal{C}_R$ is compact let 
\begin{align*}
\xi_R = \max_{(q',p', \xi') \in \mathcal{C}_R}    \xi' \qquad \text{ and } \qquad L^2_R= \max_{(q,p, \xi), (q', p', \xi') \in \mathcal{C}_R} L_{q,q'}^2  
\end{align*}
and note that by Proposition~\ref{prop:suppreg} (ii)
\begin{align*}
\mathscr{P}_{\frac{t_0}{2}}(x, B_\delta(y')) >0
\end{align*}
for all $x\in \mathcal{C}_R$, $\delta >0$ and all $y'=(q', p', \xi')\in \mathcal{X}$ with
\begin{align*}
\xi' \geq \xi_R + (a t_0/2)^{-1}  L_R^{2} - \frac{t_0}{2} a^{-1} k_B T kN .  
\end{align*}
Let $y'\in \mathcal{X}$ be any such point satisfying the above.  Clearly, there exists $z'\in \mathcal{X}$ such that $r_{t_0/2}(y', z') >0$.  Employing continuity of the density on $(0, \infty) \times \mathcal{X} \times \mathcal{X}$ and picking $\delta >0$ small enough we can ensure the following bound 
\begin{align*}
r_{\frac{t_0}{2}}(y, z) \geq \epsilon >0 
\end{align*}
for all $(y, z) \in B_\delta(y') \times B_\delta(z')$ where $\epsilon >0$ is a constant.  Hence by way of~\eqref{eqn:density} we obtain for $x\in \mathcal{C}_R$
\begin{align*}
\mathscr{P}_{t_0}(x, A) \geq \epsilon \lambda(B_\delta(z')) \,  \mathscr{P}_{\frac{t_0}{2}}(x, B_\delta(y')) \,  \frac{\lambda(A\cap B_\delta(z'))}{\lambda(B_\delta(z'))} 
\end{align*} 
where $\lambda$ denotes Lebesgue measure.  Since $x\mapsto \mathscr{P}_{\frac{t_0}{2}}(x, B_\delta(y'))$ is continuous and positive on $\mathcal{C}_R$, we infer the existence of a constant $c>0$ such that 
\begin{align*}
\inf_{x\in \mathcal{C}_R} \mathscr{P}_{t_0}(x, A) \geq c\,  \frac{\lambda(A\cap B_\delta(z'))}{\lambda(B_\delta(z'))} 
\end{align*} 
for all $A\subseteq \mathcal{X}$ Borel.  This finishes the proof. 
\end{proof}
   
We now use the previous two corollaries to conclude Theorem~\ref{thm:main}.  
\begin{proof}[Proof of Theorem~\ref{thm:main}]
We first show that the augmented Gibbs measure  $\mu$ defined in relation~\eqref{eqn:Gibbsdef} is an invariant probability measure for the Markov process $x(t)$ satisfying~\eqref{eqn:main}.  Since $U$ is a normal potential, $\mu$ is a probability measure by definition.  Note also that it is a routine calculation to check that
\begin{align*}
\mathscr{L}^* (\exp(-\beta H))=0
\end{align*}
where we recall that $\beta=1/(k_B T)$ and $\mathscr{L}^*$ is the formal $L^2$-adjoint of the generator $\mathscr{L}$.  This in turn implies that $\mu \mathscr{P}_t= \mu$ for all $t\geq 0$.  To see that $\mu$ is unique, Proposition~\ref{prop:suppreg} (i) implies that $(\mathscr{P}_t)_{t\geq 0}$ is a strong Feller Markov semigroup.  Moreover, we claim that $\supp \nu= \mathcal{X}$ for any invariant probability measure $\nu$ for the Markov process $x(t)$.  Uniqueness of $\mu$ will then follow by, for example, Theorem~3.16 of~\cite{HM_06}.   
Supposing that $\nu$ is an invariant probability measure for $x(t)$, there exists $x^*\in\mathcal{X}$ for which $x^*\in \supp \nu$.  By Proposition~\ref{prop:suppreg} (ii), for any $y\in \mathcal{X}$ we may pick $t>0$ large enough so that $y \in \supp \mathscr{P}_t(x^*, \,\cdot \,)$.  Since for any $U, V\in \mathcal{B}(\mathcal{X})$ 
\begin{align*}
\nu(U) = \int_\mathcal{X} \nu(dx) \mathscr{P}_t(x, U)\geq \int_V \nu(dx) \mathscr{P}_t(x, U)
\end{align*}  
it follows that $y\in \supp \nu$.  

To obtain the remaining conclusions in the Theorem~\ref{thm:main}, we seek to apply Theorem~1.2 of~\cite{HM_08} to the embedded Markov chain on $\mathcal{X}$ given by $\mathscr{P}_n^{t_0}:=\mathscr{P}_{n t_0}$ where $t_0>0$ is as in the statement of Corollary~\ref{cor:minor}.  Note that Corollary~\ref{cor:lyap} and Corollary~\ref{cor:minor} together imply Assumption~1 and Assumption~2 of~\cite{HM_08}.  Applying Theorem~1.2 of \cite{HM_08}, there exist constants $C>0$ and $\delta \in (0,1)$ such that 
\begin{align*}
\rho_W(\nu_1 \mathscr{P}_n^{t_0} , \nu_2 \mathscr{P}_n^{t_0}) \leq C \delta^n \rho_W(\nu_1, \nu_2)
\end{align*}  
for all $\nu_i \in \mathcal{M}_W$ and all $n\in \N \cup \{0\}$.  To reintroduce the continuous-time parameter $t>0$ in the bound above, let $t= n t_0 + \epsilon$ for some $\epsilon \in (0,t_0)$ and $n\in \N \cup \{ 0\}$.  Observe that for any $\varphi: \mathcal{X}\rightarrow \R$ measurable with $\| \varphi \|_W \leq 1$, Corollary~\ref{cor:lyap} gives
\begin{align*}
\| \mathscr{P}_\epsilon \varphi \|_W \leq \| \varphi \|_W \sup_{x\in \mathcal{X}} \frac{1+ \mathscr{P}_\epsilon W(x)}{1+ W(x)}\leq C'
\end{align*}    
for some constant $C'>0$ independent of $\epsilon >0$.  Applying Fubini-Tonelli and the Chapman-Kolmogorov equations, it then follows that 
\begin{align*}
\rho_W(\nu_1 \mathscr{P}_t, \nu_2 \mathscr{P}_t) \leq C' \rho_W(\nu_1\mathscr{P}_n^{t_0}, \nu_2\mathscr{P}_n^{t_0})\leq C C' \delta^{n} \rho_W(\nu_1, \nu_2).  
\end{align*}   
Picking $\eta= -\frac{1}{t_0} \log \delta$ and $C'' = C C'/\delta^{1/t_0}$ produces the desired estimate
\begin{align*}
\rho_W(\nu_1 \mathscr{P}_t, \nu_2 \mathscr{P}_t) \leq C'' e^{-\eta t} \rho_W(\nu_1, \nu_2) 
\end{align*}
which is satisfied for all $t\geq 0$ and $\nu_i \in \mathcal{M}_W$.    

\end{proof}

\bibliographystyle{plain}
\bibliography{NHref}

\begin{thebibliography}{10}

\bibitem{Ahn_12}
Sungjin Ahn, Anoop Korattikara, and Max Welling.
\newblock Bayesian posterior sampling via stochastic gradient fisher scoring.
\newblock {\em arXiv preprint arXiv:1206.6380}, 2012.

\bibitem{AKM12}
Avanti Athreya, Tiffany Kolba, and Jonathan~C. Mattingly.
\newblock Propagating {L}yapunov functions to prove noise-induced
  stabilization.
\newblock {\em Electron. J. Probab.}, 17:no. 96, 38, 2012.

\bibitem{Chen_14}
Tianqi Chen, Emily Fox, and Carlos Guestrin.
\newblock Stochastic gradient hamiltonian monte carlo.
\newblock In {\em International Conference on Machine Learning}, pages
  1683--1691, 2014.

\bibitem{CG_10}
Florian Conrad and Martin Grothaus.
\newblock Construction, ergodicity and rate of convergence of {$N$}-particle
  {L}angevin dynamics with singular potentials.
\newblock {\em J. Evol. Equ.}, 10(3):623--662, 2010.

\bibitem{cooke17:_geomet_ergod_two}
Ben Cooke, David~P. Herzog, Jonathan~C .~Mattingly, Scott~A. McKinley, and
  Scott~C. Schmidler.
\newblock Geometric ergodicity of two--dimensional hamiltonian systems with a
  lennard--jones--like repulsive potential.
\newblock {\em Communications in Mathematical Science}, 2017.

\bibitem{Ding_14}
Nan Ding, Youhan Fang, Ryan Babbush, Changyou Chen, Robert~D Skeel, and Hartmut
  Neven.
\newblock Bayesian sampling using stochastic gradient thermostats.
\newblock In {\em Advances in neural information processing systems}, pages
  3203--3211, 2014.

\bibitem{Dua_87}
Simon Duane, Anthony~D Kennedy, Brian~J Pendleton, and Duncan Roweth.
\newblock Hybrid monte carlo.
\newblock {\em Physics letters B}, 195(2):216--222, 1987.

\bibitem{GHM_17}
Nathan~E Glatt-Holtz, David~P Herzog, and Jonathan~C Mattingly.
\newblock Scaling and saturation in infinite-dimensional control problems with
  applications to stochastic partial differential equations.
\newblock {\em arXiv preprint arXiv:1706.01997}, 2017.

\bibitem{Grothaus_Stilgenbauer_2015}
Martin Grothaus and Patrik Stilgenbauer.
\newblock A hypocoercivity related ergodicity method with rate of convergence
  for singularly distorted degenerate kolmogorov equations and applications.
\newblock {\em Integral Equations and Operator Theory}, 83(3):331–379, Nov
  2015.
\newblock arXiv: 1506.04386.

\bibitem{HM_06}
Martin Hairer and Jonathan~C. Mattingly.
\newblock Ergodicity of the 2{D} {N}avier-{S}tokes equations with degenerate
  stochastic forcing.
\newblock {\em Ann. of Math. (2)}, 164(3):993--1032, 2006.

\bibitem{HM_08}
Martin Hairer and Jonathan~C. Mattingly.
\newblock Yet another look at harris' ergodic theorem for markov chains.
\newblock {\em Seminar on Stochastic Analysis, Random Fields and Applications
  VI: Centro Stefano Franscini, Ascona, May 2008}, pages 109--117, 2011.

\bibitem{harris_48}
Daniel~L Harris~III.
\newblock On the line-absorption coefficient due to doppler effect and damping.
\newblock {\em The Astrophysical Journal}, 108:112, 1948.

\bibitem{Her_11}
David~P. Herzog.
\newblock Geometry's fundamental role in the stability of stochastic
  differential equations.
\newblock 2011.

\bibitem{HerMat15}
David~P. Herzog and Jonathan Mattingly.
\newblock Noise-induced stabilization of planar flows i.
\newblock {\em Electron. J. Probab.}, 20:43 pp., 2015.

\bibitem{HerMat17}
David~P. Herzog and Jonathan~C. Mattingly.
\newblock Ergodicity and lyapunov functions for langevin dynamics with singular
  potentials.
\newblock {\em arXiv preprint arXiv:1711.02250}, 2017.

\bibitem{Hor_67}
Lars H\"ormander.
\newblock Hypoelliptic second order differential equations.
\newblock {\em Acta Math.}, 119:147--171, 1967.

\bibitem{Horo_91}
Alan~M Horowitz.
\newblock A generalized guided monte carlo algorithm.
\newblock {\em Physics Letters B}, 268(2):247--252, 1991.

\bibitem{hummer_62}
David~G Hummer.
\newblock Non-coherent scattering: I. the redistribution function with doppler
  broadening.
\newblock {\em Monthly Notices of the Royal Astronomical Society},
  125(1):21--37, 1962.

\bibitem{Jones_11}
Andrew Jones and Ben Leimkuhler.
\newblock Adaptive stochastic methods for sampling driven molecular systems.
\newblock {\em The Journal of chemical physics}, 135(8):084125, 2011.

\bibitem{Jur_97}
Velimir Jurdjevic.
\newblock {\em Geometric control theory}, volume~52.
\newblock Cambridge university press, 1997.

\bibitem{Kh_12}
Rafail Khasminskii.
\newblock {\em Stochastic stability of differential equations}, volume~66 of
  {\em Stochastic Modelling and Applied Probability}.
\newblock Springer, Heidelberg, second edition, 2012.
\newblock With contributions by G. N. Milstein and M. B. Nevelson.

\bibitem{leimkuhler2015computation}
Benedict Leimkuhler, Charles Matthews, and Gabriel Stoltz.
\newblock The computation of averages from equilibrium and nonequilibrium
  langevin molecular dynamics.
\newblock {\em IMA Journal of Numerical Analysis}, 36(1):13--79, 2015.

\bibitem{lether_97}
Frank~G Lether.
\newblock Constrained near-minimax rational approximations to dawson's
  integral.
\newblock {\em Applied mathematics and computation}, 88(2-3):267--274, 1997.

\bibitem{MSH_02}
J.~C. Mattingly, A.~M. Stuart, and D.~J. Higham.
\newblock Ergodicity for {SDE}s and approximations: locally {L}ipschitz vector
  fields and degenerate noise.
\newblock {\em Stochastic Process. Appl.}, 101(2):185--232, 2002.

\bibitem{mccabe_74}
JH~McCabe.
\newblock A continued fraction expansion, with a truncation error estimate, for
  dawson's integral.
\newblock {\em Mathematics of Computation}, 28(127):811--816, 1974.

\bibitem{Met_53}
Nicholas Metropolis, Arianna~W Rosenbluth, Marshall~N Rosenbluth, Augusta~H
  Teller, and Edward Teller.
\newblock Equation of state calculations by fast computing machines.
\newblock {\em The journal of chemical physics}, 21(6):1087--1092, 1953.

\bibitem{MTIII}
Sean~P. Meyn and R.~L. Tweedie.
\newblock Stability of markovian processes {III}: Foster-lyapunov criteria for
  continuous-time processes.
\newblock {\em Advances in Applied Probability}, 25(3):518--548, 1993.

\bibitem{Patterson_13}
Sam Patterson and Yee~Whye Teh.
\newblock Stochastic gradient riemannian langevin dynamics on the probability
  simplex.
\newblock In {\em Advances in Neural Information Processing Systems}, pages
  3102--3110, 2013.

\bibitem{RB_06}
Luc Rey-Bellet.
\newblock Ergodic properties of {M}arkov processes.
\newblock In {\em Open quantum systems. {II}}, volume 1881 of {\em Lecture
  Notes in Math.}, pages 1--39. Springer, Berlin, 2006.

\bibitem{sajo_93}
Erno Sajo.
\newblock On the recursive properties of dawson's integral.
\newblock {\em Journal of Physics A: Mathematical and General}, 26(12):2977,
  1993.

\bibitem{Dett_07}
Alex~A Samoletov, Carl~P Dettmann, and Mark~AJ Chaplain.
\newblock Thermostats for ``slow" configurational modes.
\newblock {\em Journal of Statistical Physics}, 128(6):1321--1336, 2007.

\bibitem{Shi_14}
Armen Shirikyan.
\newblock Approximate controllability of the viscous burgers equation on the
  real line.
\newblock In {\em Geometric control theory and sub-Riemannian geometry}, pages
  351--370. Springer, 2014.

\bibitem{SV_72}
Daniel~W. Stroock and S.~R.~S. Varadhan.
\newblock On the support of diffusion processes with applications to the strong
  maximum principle.
\newblock In {\em Proceedings of the Sixth Berkeley Symposium on Mathematical
  Statistics and Probability}, pages 333--359. Univ. California Press,
  Berkeley, Calif., 1972.

\bibitem{SV_73}
Daniel~W. Stroock and S.~R.~S. Varadhan.
\newblock Probability theory and the strong maximum principle.
\newblock In {\em Partial differential equations ({P}roc. {S}ympos. {P}ure
  {M}ath., {V}ol. {XXIII}, {U}niv. {C}alifornia, {B}erkeley, {C}alif., 1971)},
  pages 215--220. Amer. Math. Soc., Providence, R.I., 1973.

\bibitem{Talay_2002}
D.~Talay.
\newblock Stochastic hamiltonian systems: exponential convergence to the
  invariant measure, and discretization by the implicit euler scheme.
\newblock {\em Markov Processes and Related Fields}, 8(2):163--198, 2002.
\newblock Inhomogeneous random systems (Cergy-Pontoise, 2001).

\bibitem{Villani_2009}
C\'edric Villani.
\newblock Hypocoercivity.
\newblock {\em Mem. Amer. Math. Soc.}, 202(950):iv+141, 2009.

\bibitem{Welling_11}
Max Welling and Yee~W Teh.
\newblock Bayesian learning via stochastic gradient langevin dynamics.
\newblock In {\em Proceedings of the 28th International Conference on Machine
  Learning (ICML-11)}, pages 681--688, 2011.

\end{thebibliography}

\end{document}